\documentclass[12pt,a4paper,reqno]{amsart}
\usepackage{amssymb}
\usepackage{amscd}
\usepackage{enumerate}
\usepackage{graphicx}
\usepackage{siunitx}
\usepackage{tikz-cd}
\usepackage{color}
\usetikzlibrary{arrows}
\numberwithin{equation}{section}

\usepackage{mathabx}

\usepackage{mathtools}
\usepackage[tableposition=top]{caption}
\usepackage{booktabs,dcolumn}

\DeclareFontFamily{OT1}{rsfs}{}
\DeclareFontShape{OT1}{rsfs}{n}{it}{<-> rsfs10}{}
\DeclareMathAlphabet{\mathscr}{OT1}{rsfs}{n}{it}

\theoremstyle{plain}

\newtheorem{theorem}{Theorem}[section]
\newtheorem{proposition}[theorem]{Proposition}
\newtheorem{lemma}[theorem]{Lemma}

\theoremstyle{definition}

\newtheorem{definition}[theorem]{Definition}
\newtheorem{remark}[theorem]{Remark}

\usepackage{algorithm, algorithmic}

\begin{document}

\title[Multiple solutions with constant sign]{Multiple solutions with constant sign of a Dirichlet problem for a class of
elliptic systems with variable exponent growth}

\author{Li Yin}
\address{College of Information and Management Science, Henan Agricultural University,
Zhengzhou, Henan 450002, China}
\email{mathsr@163.com (L. Yin)}

\author{Jinghua Yao$^\dagger$}
\address{$^\dagger$ \textbf{Corresponding author}, Department of Mathematics, Indiana University, Bloomington, IN, 47408, USA.}
\email{yaoj@indiana.edu (J. Yao)}

\author{Qihu Zhang}
\address{Department of Mathematics and Information Science, Zhengzhou
University of Light Industry, Zhengzhou, Henan, 450002, China.}
\email{zhangqihu@yahoo.com (Q. Zhang)}

\author{Chunshan Zhao}
\address{Department of Mathematical Sciences, Georgia Southern University, Statesboro, GA 30460, USA.}
\email{czhao@GeorgiaSouthern.edu (C. Zhao)}
\date{\today}

\thanks{$\dagger$ Corresponding author}
\thanks{$^\star$This research is partly supported by
the key projects in Science and Technology Research of the Henan Education
Department (14A110011).}


\begin{abstract}
We investigate the following
Dirichlet problem with variable exponents:
\begin{equation*}
\left\{
\begin{array}{l}
-\bigtriangleup _{p(x)}u=\lambda \alpha (x)\left\vert u\right\vert ^{\alpha
(x)-2}u\left\vert v\right\vert ^{\beta (x)}+F_{u}(x,u,v),\text{ in }\Omega ,
\\
-\bigtriangleup _{q(x)}v=\lambda \beta (x)\left\vert u\right\vert ^{\alpha
(x)}\left\vert v\right\vert ^{\beta (x)-2}v+F_{v}(x,u,v),\text{ in }\Omega ,
\\
u=0=v,\text{ on }\partial \Omega.
\end{array}
\right.
\end{equation*}
We present here, in the system setting, a new set of growth conditions under which we manage to use a novel method to verify the Cerami compactness condition. By localization argument, decomposition technique and variational methods, we are able to show the existence of multiple solutions with constant sign for the problem without the well-known
Ambrosetti--Rabinowitz type growth condition. More precisely, we manage to show that the problem admits
four, six and infinitely many solutions respectively.


\end{abstract}

\maketitle

\baselineskip 14pt

\section{ Introduction}

In this paper, we consider the existence of multiple solutions to the
following Dirichlet problem for an elliptic system with variable exponents:
\begin{equation*}
\text{$(P)$ }\left\{
\begin{array}{l}
-\bigtriangleup _{p(x)}u=\lambda \alpha (x)\left\vert u\right\vert ^{\alpha
(x)-2}u\left\vert v\right\vert ^{\beta (x)}+F_{u}(x,u,v),\text{ in }\Omega ,
\\
-\bigtriangleup _{q(x)}v=\lambda \beta (x)\left\vert u\right\vert ^{\alpha
(x)}\left\vert v\right\vert ^{\beta (x)-2}v+F_{v}(x,u,v),\text{ in }\Omega ,
\\
u=0=v,\text{ on }\partial \Omega,
\end{array}
\right.
\end{equation*}
where $\bigtriangleup _{p(x)}u:=\mbox{div}(\left\vert \nabla u\right\vert
^{p(x)-2}\nabla u)$ is called $p(x)$-Laplacian which is nonlinear and
nonhomogeneous, $\Omega \subset \mathbb{R}^{N}$ is a bounded domain, and $
p(\cdot ),q(\cdot ),\alpha (\cdot ),\beta (\cdot )>1$ are in the space $
C^{1}(\overline{\Omega })$ which consists of differentiable functions with
continuous first order derivatives on $\overline{\Omega }$.

Elliptic equations and systems of elliptic equations with variable exponent
growth as in problem $(P)$ arise from applications in electrorheological
fluids and image restoration. We refer the readers to \cite{e1}, \cite{e2},
\cite{e3}, \cite{e4} and the references therein for more details in
applications. In particular, see \cite{e2} for a model with variable
exponent growth and its important applications in image denoising,
enhancement, and restoration. Problems with variable exponent growth also
brought challenging pure mathematical problems. Compared with the classical
Laplacian $\Delta=\Delta_2$ which is linear and homogeneous and the $p$
-Laplacian $\Delta_p\cdot:=\mbox{div}(|\nabla\cdot|^{p-2}\nabla \cdot)$
which is nonlinear but homogeneous for constant $p$, the $p(x)$-Laplacian is
both nonlinear and inhomogeneous. Due to the nonlinear and inhomogeneous
nature of the $p(x)$-Laplacian, nonlinear (systems of) elliptic equations
involving $p(x)$-Laplacian and nonlinearities with variable growth rates are
much more difficult to deal with. Driven by the real-world applications and
mathematical challenges, the study of elliptic equations and systems with
variable exponent growth has attracted many researchers with different
backgrounds, and become a very attracting field. We refer the readers to
\cite{j8}, \cite{e12}, \cite{e13}, \cite{e16}, \cite{e17}, \cite{j1}, \cite
{zj2}, \cite{rr}, \cite{e30}, \cite{e31}, \cite{zj3} and the related
references to tract the rapid development of the field.

In this paper, our main goal is to obtain some existence results for the
problem $(P)$, a Dirichlet problem for elliptic systems with variable
exponents, without the famous Ambrosetti-Rabinowitz condition via critical
point theory. For this purpose, we propose a new set of growth conditions
for the nonlinearities in the current system of elliptic equations setting.
Our new set of growth conditions involve only variable growths which
naturally match the variable nature of the problem under investigation.
Under our growth conditions, we can use a novel method to verify that the
corresponding functional to the problem $(P)$ satisfies the Cerami
compactness condition which is a weaker compactness condition yet is still
sufficient to yield critical points of the functional. See the details of
proofs in Section 3. The current study generalizes in particular our former
investigations \cite{yyzz} and \cite{zj3}. However, this generalization from
a single elliptic equation to the current system of elliptic equations is by
no means trivial. Besides technical complexities, the assumptions in the
current study are more involved. In particular, though we still do not need
any monotonicity on the nonlinear terms, we do need impose certain
monotonicity assumptions on the variable exponents to close our argument in
the system setting.

When we utilize variational argument to obtain existence of weak solutions
to elliptic equations, typically we impose the famous Ambrosetti-Rabinowitz
growth condition on the nonlinearity to guarantee the boundedness of
Palais-Samle sequence. Under the Ambrosetti-Rabinowitz growth condition, one
then tries to verify the Palais-Smale condition. However, the
Ambrosetti-Rabinowitz type growth condition excludes a number of interesting
nonlinearities. In view of this fact, a lot of efforts were made to show the
existence of weak solutions in the variational framework without this type
of growth condition, especially for the usual $p$-Laplacian and a single
nonlinear elliptic partial differential equation (see, in particular, \cite
{zj4}, \cite{zj5}, \cite{jj5}, \cite{jj3}, \cite{jj2}, \cite{jj4} and the
references therein). Our results can be regarded as extensions of the
corresponding results for the $p$-Laplacian problems. There are also some
related earlier works which dealt with elliptic variational problems in the
variable exponent spaces framework, see \cite{jj1}, \cite{zj1}, \cite{31a},
\cite{zj3}, \cite{wyl} and related works. These earlier studies were mainly
focused on a single elliptic equation. In the interesting earlier study \cite
{31a}, the author considered the existence of solutions of the following
variable exponent differential equations without Ambrosetti-Rabinowitz
condition on bounded domain,
\begin{equation}
\left\{
\begin{array}{l}
-\bigtriangleup _{p(x)}u=f(x,u)\text{, in }\Omega \text{,} \\
u=0\text{, on }\partial \Omega.
\end{array}
\right.  \label{aa2}
\end{equation}
However, in some aspects the assumption is even stronger than the
Ambrosetti-Rabinowitz condition. In a recent study \cite{jj1}, the authors
considered the variable exponent equation in the whole space $\mathbb{R}^N$
under the following assumptions: $(1^{0})$ there exists a constant $\theta
\geq 1$, such that $\theta \mathcal{F}(x,t)\geq \mathcal{F}(x,st)$ for any$\
(x,t)\in \mathbb{R}^{N}\mathbb{\times R}$ and $s\in \lbrack 0,1]$, where $
\mathcal{F}(x,t)=f(x,t)t-p^{+}F(x,t)$; $(2^{0})$ $f\in C(\mathbb{R}^{N}
\mathbb{\times R},\mathbb{R})$ satisfies $\text{ }\underset{\left\vert
t\right\vert \rightarrow \infty }{\lim }\frac{F(x,t)}{\left\vert
t\right\vert ^{p^{+}}}=\infty . \label{abc} $ In \cite{zj1}, the authors
considered the problem (\ref{aa2}) in a bounded domain under the condition $
(2^{0})$. In \cite{wyl}, the authors studied a variable exponent
differential equation with a potential term in the whole space $\mathbb{R}^N$
. The authors proposed conditions under which they could show the existence
of infinitely many high energy solutions without the Ambrosetti-Rabinowitz
condition. In the above mentioned works, the growths conditions involved
either the supremum or the infimum of the variable exponents. In our current
study, we are able to provide a number of existence results in the system
setting under assumptions that only involve variable exponent growths which
match naturally the variable growth of the problem under study.

We point out here that the growth conditions we use here and the method to
check the Cerami compactness condition are different from all the above
mentioned works. Due to the differences between the $p$-Laplacian and $p(x)$
-Laplacian mentioned as above, it is usually challenging to judge whether or
not results about $p$-Laplacian can be generalized to $p(x)$-Laplacian.
Meanwhile, some new methods and techniques are needed to study elliptic
equations involving the non-standard growth, as the commonly known methods
and techniques to study elliptic equations involving standard growth may
fail. The main reason, as mentioned earlier, is that the principal elliptic
operators in the elliptic equations involving the non-standard growth is not
homogeneous anymore. To see some new features associated with the $p(x)$
-Laplacian, we first point out that the norms in variable exponent spaces
are the so-called Luxemburg norms $\left\vert u\right\vert _{p(\cdot )}$
(see Section 2) and the integral $\int_{\Omega }\left\vert u(x)\right\vert
^{p(x)}dx$ does not have the usual constant power relation as in the spaces $
L^{p}$ for constants $p$. Another subtle feature is on the principal
Dirichlet eigenvalue. As invetigated in \cite{e13}, even for a bounded
smooth domain $\Omega \subset \mathbb{R}^{N}$, the principle eigenvalue $
\lambda _{p(\cdot )}$ defined by the Rayleigh quotient
\begin{equation}
\lambda _{p(\cdot )}=\underset{u\in W_{0}^{1,p(\cdot )}(\Omega )\backslash
\{0\}}{\inf }\frac{\int_{\Omega }\frac{1}{p(x)}\left\vert \nabla
u\right\vert ^{p(x)}dx}{\int_{\Omega }\frac{1}{p(x)}\left\vert u\right\vert
^{p(x)}dx}  \label{d2}
\end{equation}
is zero in general, and only under some special conditions $\lambda
_{p(\cdot )}>0$ holds. For example, when $\Omega \subset \mathbb{R}$ ($N=1$)
is an interval, results show that $\lambda _{p(\cdot )}$ $>0$ if and only if
$p(\cdot )$ is monotone. This feature on the $p(x)$-Laplacian Dirichlet
principle eigenvalue plays an important role for us in proposing the
assumptions on the variable exponents and in our proofs of the main results.

Now we shall first list the assumptions on the nonlinearity $F$ and variable
exponents involved in the current system setting. Our assumptions are as
follows.

$(H_{\alpha ,\beta })$ $\frac{\alpha (x)}{p(x)}+\frac{\beta (x)}{q(x)}
<1,\forall x\in \overline{\Omega }$.

$(H_{0})$ $F:\Omega \times \mathbb{R}\times \mathbb{R}\rightarrow \mathbb{R}
$ is $C^{1}$ continous and
\begin{equation*}
\left\vert F_{u}(x,u,v)u\right\vert +\left\vert F_{v}(x,u,v)v\right\vert
\leq C(1+\left\vert u\right\vert ^{\gamma (x)}+\left\vert v\right\vert
^{\delta (x)}),\forall (x,u,v)\in \Omega \times \mathbb{R},
\end{equation*}
where $\gamma ,\delta \in C(\overline{\Omega })$ and $p(x)<\gamma
(x)<p^{\ast }(x),q(x)<\delta (x)<q^{\ast }(x)$ where $p^{\ast }(x)$ and $
q^{\ast }(x)$ are defined as
\begin{equation*}
p^{\ast }(x)=
\begin{cases}
\frac{Np(x)}{N-p(x)},\,\,\mbox{if}\,\, p(x)<N \\
+\infty, \,\,\quad\,\mbox{if}\,\, p(x)\geq N
\end{cases}
,\,\, q^{\ast }(x)=
\begin{cases}
\frac{Nq(x)}{N-q(x)},\,\,\mbox{if}\,\, q(x)<N \\
+\infty, \,\,\quad\,\mbox{if}\,\, q(x)\geq N
\end{cases}
\end{equation*}

$(H_{1})$ There exists constants $M,C_{1},C_{2}>0$, $a(\cdot )>p(\cdot )$
and $b(\cdot )>q(\cdot )$ on $\overline{\Omega }$ such that
\begin{eqnarray*}
&&C_{1}\left\vert u\right\vert ^{p(x)}[\ln (e+\left\vert u\right\vert
)]^{a(x)-1}+C_{1}\left\vert v\right\vert ^{q(x)}[\ln (e+\left\vert
v\right\vert )]^{b(x)-1} \\
&\leq &C_{2}\left( \frac{F_{u}(x,u,v)u}{\ln (e+\left\vert u\right\vert )}+
\frac{F_{v}(x,u,v)v}{\ln (e+\left\vert v\right\vert )}\right) \\
&\leq &\frac{1}{p(x)}F_{u}(x,u,v)u+\frac{1}{q(x)}F_{v}(x,u,v)v-F(x,u,v),
\forall \left\vert u\right\vert +\left\vert v\right\vert \geq M,\forall x\in
\Omega .
\end{eqnarray*}

$(H_{2})$ $F(x,u,v)=o(\left\vert u\right\vert ^{p(x)}+\left\vert
v\right\vert ^{q(x)})$ uniformly for $x\in \Omega $ as $u,v\rightarrow 0$.

$(H_{3})$ $F$ satisfies $F_{u}(x,0,v)=0,F_{v}(x,u,0)=0$, $\forall x\in
\overline{\Omega }$, $\forall u,v\in\mathbb{R}$.

$(H_{4})$ $F(x,-u,-v)=F(x,u,v),$ $\forall x\in \overline{\Omega },$ $\forall
u,v\in \mathbb{R}$.

$(H_{p,q})$ There are vectors $l_{p},l_{q}\in \mathbb{R}^{N}\backslash \{0\}
$ such that for any $x\in \Omega $, $\phi _{p}(t)=p(x+tl_{p})$ is monotone
for $t\in I_{x,p}(l)=\{t\mid x+tl_{p}\in \Omega \}$, and $\phi
_{q}(t)=q(x+tl_{q})$ is monotone for $t\in I_{x,q}(l)=\{t\mid x+tl_{q}\in
\Omega \}$.

To gain a first understanding, we briefly comment on some of the above
assumptions. $(H_{0})$ means that the nonlinearity $F$ has a subcritical
variable growth rate in the sense of variable exponent Sobolev embedding and
in the current system of elliptic equations setting. $(H_{1})$ and $(H_{2})$
describe the far and near field behaviors of the nonlinearity $F$. Notice
that in the current setting, the far field behavior $(H_{1})$ is more
involved. We emphasize that $(H_{p,q})$ is crucial for our later arguments,
for it guarantees that the Rayleigh quotients for $-\Delta_{p(x)}$ and $
-\Delta_{q(x)}$ (see \eqref{d2} for $-\Delta_{p(x)}$) are positive. Finally,
the assumptions $(H_0)$-$(H_4)$ on the nonlinearity $F$ are consistent,
which can be seen by the following example:
\begin{equation*}
F(x,u,v)=\left\vert u\right\vert ^{p(x)}[\ln (1+\left\vert u\right\vert
)]^{a(x)}+\left\vert v\right\vert ^{q(x)}[\ln (1+\left\vert v\right\vert
)]^{b(x)}+\left\vert u\right\vert ^{\theta _{1}(x)}\left\vert v\right\vert
^{\theta _{2}(x)}\ln (1+\left\vert u\right\vert )\ln (1+\left\vert
v\right\vert ),
\end{equation*}
where $1<\theta _{1}(x)<p(x),1<\theta _{2}(x)<q(x),\frac{\theta _{1}(x)}{p(x)
}+\frac{\theta _{2}(x)}{q(x)}=1,\forall x\in \overline{\Omega }$. In
addition, $F$ does not satisfy the Ambrosetti-Rabinowitz condition.

Now we are in a position to state our main results.

\begin{theorem}\label{thm1}
If $\lambda $ is small enough and the assumptions (H$_{\alpha
,\beta }$), $(H_{0})$, (H$_{2}$)-(H$_{3}$) and (H$_{p,q}$) hold, then the
problem $(P)$ has at least four nontrivial solutions each with constant sign
respectively.
\end{theorem}

\begin{theorem}\label{thm2} 
If $\lambda $ is small enough and the assumptions $(H_{\alpha
,\beta })$, $(H_{0})$-$(H_{3})$ and $(H_{p,q})$ hold, then the problem $(P)$
has at least six nontrivial solutions each with constant sign respectively.
\end{theorem}

\begin{theorem}\label{thm3} 
If the assumptions $(H_{\alpha ,\beta })$, $(H_{0})$, $(H_{1})$
and $(H_{4})$ hold, then there are infinitely many pairs of solutions to the
problem $(P)$.
\end{theorem}

\begin{remark}
The solutions we obtained in Theorem \ref{thm1} and \ref{thm2} to $(P)$ are
of constant sign. Meanwhile, we do not need any monotonicity assumptions on
the nonlinearity $F(x,\cdot ,\cdot )$.
\end{remark}

This rest of the paper is organized as follows. In Section 2, we do some
functional-analytic preparations. In Section 3, we give the proofs of our
main results.

\section{Functional-analytic Preliminary}

Throughout this paper, we will use letters $c,c_{i},C,C_{i}$, $i=1,2,...$ to
denote generic positive constants which may vary from line to line, and we
will specify them whenever it is necessary.

In order to discuss problem $(P)$, we shall discuss the functional analytic
framework. First, we present some results about space $W_{0}^{1,p(\cdot
)}(\Omega )$ which we call variable exponent Sobolev space. These results on
the variable exponent spaces will be used later (for details, see \cite{j5},
\cite{e9}, \cite{e12}, \cite{e20}, \cite{e26}). We denote $
C(\overline\Omega) $ the space of continuous functions on $\overline\Omega$
with the usual uniform norm, and
\begin{equation*}
C_{+}(\overline{\Omega }) =\left\{ h\left\vert h\in C(\overline{\Omega })
\text{, }h(x)>1\text{ for }x\in \overline{\Omega }\right. \right\}.
\end{equation*}
For $h=h(\cdot)\in C(\Omega)$, we denote $h^+:=\max_{x\in\bar{\Omega}}\,
h(x) $ and $h^-:=\min_{x\in\bar{\Omega}}\, h(x)$. For $p=p(\cdot)\in
C_{+}(\overline\Omega)$, we introduce
\begin{equation*}
L^{p(\cdot )}(\Omega ) =\left\{ u\mid u\text{ is a measurable real-value
function, }\int_{\Omega }\left\vert u(x)\right\vert ^{p(x)}dx<\infty
\right\} .
\end{equation*}

When equipped with the Luxemberg norm

\begin{center}
$\left\vert u\right\vert _{p(\cdot )}=$ $\inf \left\{ \mu >0\left\vert
\int_{\Omega }\left\vert \frac{u(x)}{\mu }\right\vert ^{p(x)}dx\leq 1\right.
\right\} $,
\end{center}

($L^{p(\cdot )}(\Omega )$, $\left\vert \cdot \right\vert _{p(\cdot )}$)
becomes a Banach space and it is called variable exponent Lebesgue space.
For the variable exponent Lebesgue spaces, we have the following H$\ddot{o}$lder type inequality and simple embedding relation.

\begin{proposition}
(see \cite{j5}, \cite{e9}, \cite{e12}). i) The space $(L^{p(\cdot )}(\Omega
),\left\vert \cdot \right\vert _{p(\cdot )})$ is a separable, uniform convex
Banach space, and its conjugate space is $L^{(p(\cdot ))^{0}}(\Omega )$,
where $(p(\cdot ))^{0}:=\frac{p(\cdot )}{p(\cdot )-1}$ is the conjugate
function of $p(\cdot )$. For any $u\in L^{p(\cdot )}(\Omega )$ and $v\in
L^{(p(\cdot ))^{0}}(\Omega )$, we have
\begin{equation*}
\left\vert \int_{\Omega }uvdx\right\vert \leq (\frac{1}{p^{-}}+\frac{1}{q^{-}
})\left\vert u\right\vert _{p(\cdot )}\left\vert v\right\vert _{(p(\cdot
))^{0}}.
\end{equation*}

ii) If $p_{1},$ $p_{2}\in C_{+}(\overline{\Omega })$, $p_{1}(x)\leq p_{2}(x)$
for any $x\in \overline{\Omega },$ then $L^{p_{2}(\cdot )}(\Omega )\subset
L^{p_{1}(\cdot )}(\Omega ),$ and the imbedding is continuous.
\end{proposition}

Denote $Y=\underset{i=1}{\overset{k}{\prod }}L^{p_{i}(\cdot )}(\Omega )$
with the norm
\begin{equation*}
\parallel y\parallel _{Y}=\overset{k}{\underset{i=1}{\sum }}\left\vert
y^{i}\right\vert _{p_{i}(\cdot )}\text{, }\forall y=(y^{1},\cdots ,y^{k})\in
Y\text{,}
\end{equation*}
where $p_{i}(x)\in C_{+}(\overline{\Omega })$, $i=1,\cdots ,m$, then $Y$ is
a Banach space. The following proposition can be regarded as a vectorial
generalization of the classical proposition on the Nemytsky operator.

\begin{proposition}
\label{2.2} Let $f(x,y):$ $\Omega \times\mathbb{R}
^{k}\rightarrow\mathbb{R}
^{m}$ be a Caratheodory function, i.e., $f$ satisfies

(i) for a.e. $x\in \Omega $, $y\rightarrow f(x,y)$ is a continuous function
from $\mathbb{R}^{k}\ $\ to $\mathbb{R}^{m}$,

(ii) for any $y\in\mathbb{R}^{k}$, $x\rightarrow f(x,y)$ is measurable.

If there exist $\eta (x),p_{1}(x),\cdots ,p_{k}(x)\in C_{+}(\overline{\Omega
})$, $h(x)\in L^{\eta (\cdot )}(\Omega )$ and positive constant $c>0$ such
that
\begin{equation*}
\left\vert f(x,y)\right\vert \leq h(x)+c\underset{i=1}{\overset{k}{\sum }}
\left\vert y_{i}\right\vert ^{p_{i}(x)/\eta (x)}\text{ for any }x\in \Omega
,y\in\mathbb{R}^{k},
\end{equation*}
then the Nemytsky operator from $Y$ to $(L^{\eta (\cdot )}(\Omega ))^{m}$
defined by $(N_{f}u)(x)=f(x,u(x))$ is a continuous and bounded operator.
\end{proposition}

\begin{proof}
Similar to the proof of \cite{e39}, we omit it here.
\end{proof}

The following two propositions concern the norm-module relations in the
variable exponent Lebesgue spaces. Unlike in the usual Lebesgue spaces
setting, the norm and module of a function in the variable exponent spaces
do not enjoy the usual power equality relation.

\begin{proposition}
(see \cite{e12}). If we denote
\begin{equation*}
\rho (u)=\int_{\Omega }\left\vert u\right\vert ^{p(x)}dx\text{, }\forall
u\in L^{p(\cdot )}(\Omega ),
\end{equation*}
then there exists a $\xi \in \overline{\Omega }$ such that $\left\vert
u\right\vert _{p(\cdot )}^{p(\xi )}=\int_{\Omega }\left\vert u\right\vert
^{p(x)}dx$ and

i) $\left\vert u\right\vert _{p(\cdot )}<1(=1;>1)\Longleftrightarrow \rho
(u)<1(=1;>1);$

ii) $\left\vert u\right\vert _{p(\cdot )}>1\Longrightarrow \left\vert
u\right\vert _{p(\cdot )}^{p^{-}}\leq \rho (u)\leq \left\vert u\right\vert
_{p(\cdot )}^{p^{+}};$ $\left\vert u\right\vert _{p(\cdot
)}<1\Longrightarrow \left\vert u\right\vert _{p(\cdot )}^{p^{-}}\geq \rho
(u)\geq \left\vert u\right\vert _{p(\cdot )}^{p^{+}};$

iii) $\left\vert u\right\vert _{p(\cdot )}\rightarrow 0\Longleftrightarrow
\rho (u)\rightarrow 0;$ $\left\vert u\right\vert _{p(\cdot )}\rightarrow
\infty \Longleftrightarrow \rho (u)\rightarrow \infty $.
\end{proposition}

\begin{proposition}
(see \cite{e12}). If $u$, $u_{n}\in L^{p(\cdot )}(\Omega )$, $n=1,2,\cdots ,$
then the following statements are equivalent to each other.

1) $\underset{k\rightarrow \infty }{\lim }$ $\left\vert u_{k}-u\right\vert
_{p(\cdot )}=0;$

2) $\underset{k\rightarrow \infty }{\lim }$ $\rho \left( u_{k}-u\right) =0;$

3) $u_{k}$ $\rightarrow $ $u$ in measure in $\Omega $ and $\underset{
k\rightarrow \infty }{\lim }$ $\rho \left( u_{k}\right) =\rho (u).$
\end{proposition}

The spaces $W^{1,p(\cdot )}(\Omega )$ and $W^{1,q(\cdot )}(\Omega )$ are
defined by
\begin{eqnarray*}
W^{1,p(\cdot )}(\Omega ) &=&\left\{ u\in L^{p(\cdot )}\left( \Omega \right)
\left\vert \nabla u\in (L^{p(\cdot )}\left( \Omega \right) )^{N}\right.
\right\} , \\
W^{1,q(\cdot )}(\Omega ) &=&\left\{ v\in L^{q(\cdot )}\left( \Omega \right)
\left\vert \nabla v\in (L^{q(\cdot )}\left( \Omega \right) )^{N}\right.
\right\} ,
\end{eqnarray*}
and be endowed with the following norm
\begin{eqnarray*}
\left\Vert u\right\Vert _{p(\cdot )} &=&\inf \left\{ \mu >0\left\vert
\int_{\Omega }\left\vert \frac{\nabla u}{\mu }\right\vert
^{p(x)}dx+\int_{\Omega }\left\vert \frac{u(x)}{\mu }\right\vert
^{p(x)}dx\leq 1\right. \right\} , \\
\left\Vert u\right\Vert _{q(\cdot )} &=&\inf \left\{ \mu >0\left\vert
\int_{\Omega }\left\vert \frac{\nabla u}{\mu }\right\vert
^{q(x)}dx+\int_{\Omega }\left\vert \frac{u(x)}{\mu }\right\vert
^{q(x)}dx\leq 1\right. \right\} .
\end{eqnarray*}

We denote by $W_{0}^{1,p(\cdot )}(\Omega )$ the closure of $C_{0}^{\infty
}\left( \Omega \right) $ in $W^{1,p(\cdot )}(\Omega )$. Then we have in
particular the following Sobolev embedding relation and Poincar$\acute{e}$
type inequality.

\begin{proposition}
\label{2.5} (see \cite{j5}, \cite{e12}). i) $W^{1,p(\cdot )}(\Omega )$ and $
W_{0}^{1,p(\cdot )}(\Omega )$ are separable reflexive Banach spaces;

ii) If $\eta \in C_{+}\left( \overline{\Omega }\right) $ and $\eta
(x)<p^{\ast }(x)$ for any $x\in \overline{\Omega },$ then the imbedding from
$W^{1,p(\cdot )}(\Omega )$ to $L^{\eta (\cdot )}\left( \Omega \right) $ is
compact and continuous;

iii) There is a constant $C>0,$ such that
\begin{equation*}
\left\vert u\right\vert _{p(\cdot )}\leq C\left\vert \nabla u\right\vert
_{p(\cdot )},\forall u\in W_{0}^{1,p(\cdot )}(\Omega ).
\end{equation*}
\end{proposition}

We know from iii) of Proposition \ref{2.5} that $\left\vert \nabla
u\right\vert _{p(\cdot )}$ and $\left\Vert u\right\Vert _{p(\cdot )}$ are
equivalent norms on $W_{0}^{1,p(\cdot )}(\Omega )$. From now on we will use $
\left\vert \nabla u\right\vert _{p(\cdot )}$ to replace $\left\Vert
u\right\Vert _{p(\cdot )}$ as the norm on $W_{0}^{1,p(\cdot )}(\Omega )$,
and use $\left\vert \nabla v\right\vert _{q(\cdot )}$ to replace $\left\Vert
v\right\Vert _{q(\cdot )}$ as the norm on $W_{0}^{1,q(\cdot )}(\Omega )$.

Under the assumption (H$_{p,q}$), $\lambda _{p(\cdot )}$ defined in (\ref{d2}
) is positive, i.e., we have the following proposition.

\begin{proposition}
(see \cite{e13}). If the assumption (H$_{p,q}$) is satisfied, then $\lambda
_{p(\cdot )}$ defined in (\ref{d2}) is positive.
\end{proposition}

\bigskip Denote $X=W_{0}^{1,p(\cdot )}(\Omega )\times W_{0}^{1,q(\cdot
)}(\Omega )$. The norm $\left\Vert \cdot \right\Vert $ on $X$ is defined by
\begin{equation*}
\left\Vert (u,v)\right\Vert =\max \{\left\Vert u\right\Vert _{p(\cdot
)},\left\Vert v\right\Vert _{q(\cdot )}\}.
\end{equation*}

For any $(u,v)$ and $(\phi ,\psi )$ in $X$, let
\begin{eqnarray*}
\Phi _{1}(u) &=&\int_{\Omega }\frac{1}{p(x)}\left\vert \nabla u\right\vert
^{p(x)}dx,\text{ } \\
\Phi _{2}(v) &=&\int_{\Omega }\frac{1}{q(x)}\left\vert \nabla v\right\vert
^{q(x)}dx, \\
\Phi (u,v) &=&\Phi _{1}(u)+\Phi _{2}(v), \\
\Psi (u,v) &=&\int_{\Omega }\lambda \left\vert u\right\vert ^{\alpha
(x)}\left\vert v\right\vert ^{\beta (x)}+F(x,u,v)dx.
\end{eqnarray*}

From Proposition \ref{2.2}, Proposition \ref{2.5} and condition $(H_{0})$,
it is easy to see that $\Phi _{1},\Phi _{2},\Phi ,\Psi \in C^{1}(X,\mathbb{R}
)$ and then
\begin{eqnarray*}
\Phi ^{\prime }(u,v)(\phi ,\psi ) &=&D_{1}\Phi (u,v)(\phi )+D_{2}\Phi
(u,v)(\psi ), \\
\Psi ^{\prime }(u,v)(\phi ,\psi ) &=&D_{1}\Psi (u,v)(\phi )+D_{2}\Psi
(u,v)(\psi ),
\end{eqnarray*}
where
\begin{eqnarray*}
D_{1}\Phi (u,v)(\phi ) &=&\int_{\Omega }\left\vert \nabla u\right\vert
^{p(x)-2}\nabla u\nabla \phi dx=\Phi _{1}^{\prime }(u)(\phi ), \\
D_{2}\Phi (u,v)(\psi ) &=&\int_{\Omega }\left\vert \nabla v\right\vert
^{q(x)-2}\nabla v\nabla \psi dx=\Phi _{2}^{\prime }(v)(\psi ), \\
D_{1}\Psi (u,v)(\phi ) &=&\int_{\Omega }\lambda \alpha (x)\left\vert
u\right\vert ^{\alpha (x)-2}u\left\vert v\right\vert ^{\beta (x)}+\frac{
\partial }{\partial u}F(x,u,v)\phi dx,\text{ } \\
D_{2}\Psi (u,v)(\psi ) &=&\int_{\Omega }\lambda \beta (x)\left\vert
u\right\vert ^{\alpha (x)}\left\vert v\right\vert ^{\beta (x)-2}v+\frac{
\partial }{\partial v}F(x,u,v)\psi dx.
\end{eqnarray*}
The integral functional associated with the problem $(P)$ is
\begin{equation*}
\varphi (u,v)=\Phi (u,v)-\Psi (u,v).
\end{equation*}

Without loss of generality, we may assume that $F(x,0,0)=0,\forall x\in
\overline{\Omega }$. Obviously, we have
\begin{equation*}
F(x,u,v)=\int_{0}^{1}[u\partial _{2}F(x,tu,tv)+v\partial
_{3}F(x,tu,tv)]dt,\forall x\in \overline{\Omega },
\end{equation*}
where $\partial _{j}$ denotes the partial derivative of $F$ with respect to
its $j$-th variable, then the condition $(H_{0})$ holds
\begin{equation}\label{a.2}
|F(x,u,v)|\leq c(|u|^{\gamma (x)}+|v|^{\delta (x)}+1),\forall x\in \overline{
\Omega }.
\end{equation}

From Proposition \ref{2.2}, Proposition \ref{2.5} and condition $(H_{0})$,
it is easy to see that $\varphi \in C^{1}(X,\mathbb{R})$ and satisfies
\begin{equation*}
\varphi ^{\prime }(u,v)(\phi ,\psi )=D_{1}\varphi (u,v)(\phi )+D_{2}\varphi
(u,v)(\psi ),
\end{equation*}
where
\begin{eqnarray*}
D_{1}\varphi (u,v)(\phi ) &=&D_{1}\Phi (u,v)(\phi )-D_{1}\Psi (u,v)(\phi ),
\\
D_{2}\varphi (u,v)(\psi ) &=&D_{2}\Phi (u,v)(\psi )-D_{2}\Psi (u,v)(\psi ).
\end{eqnarray*}

We say $(u,v)\in X$ is a critical point of $\varphi $ if
\begin{equation*}
\varphi ^{\prime }(u,v)(\phi ,\psi )=0,\forall (\phi ,\psi )\in X\text{.}
\end{equation*}

The dual space of $X$ will be denoted as $X^{\ast }$, then for any $H\in
X^{\ast }$, there exists $f\in (W_{0}^{1,p(\cdot )}(\Omega ))^{\ast }$, $
g\in (W_{0}^{1,q(\cdot )}(\Omega ))^{\ast }$ such that $H(u,v)=f(u)+g(v)$.
We denote $\left\Vert \cdot \right\Vert _{\ast },$ $\left\Vert \cdot
\right\Vert _{\ast ,p(\cdot )}$ and $\left\Vert \cdot \right\Vert _{\ast
,q(\cdot )}$ \ the norms of $X^{\ast },(W_{0}^{1,p(\cdot )}(\Omega ))^{\ast
}\ $and $(W_{0}^{1,q(\cdot )}(\Omega ))^{\ast }$, respectively. Holds $
X^{\ast }=(W_{0}^{1,p(\cdot )}(\Omega ))^{\ast }\times (W_{0}^{1,q(\cdot
)}(\Omega ))^{\ast }$, and
\begin{equation*}
\left\Vert H\right\Vert _{\ast }=\left\Vert f\right\Vert _{\ast ,p(\cdot
)}+\left\Vert g\right\Vert _{\ast ,q(\cdot )}.
\end{equation*}
Therefore
\begin{equation*}
\left\Vert \varphi ^{\prime }(u,v)\right\Vert _{\ast }=\left\Vert
D_{1}\varphi (u,v)\right\Vert _{\ast ,p(\cdot )}+\left\Vert D_{2}\varphi
(u,v)\right\Vert _{\ast ,q(\cdot )}.
\end{equation*}

It's easy to see that $\Phi $ is a convex functional, and we have the
following proposition.

\begin{proposition}
(see \cite{e12}, \cite{e17}). i) $\Phi ^{\prime }:X\rightarrow X^{\ast }$ is
a continuous, bounded and strictly monotone operator;

ii) $\Phi ^{\prime }$ is a mapping of type $(S_{+})$, i.e., if $
(u_{n},v_{n})\rightharpoonup (u,v)$ in $X$ and $\underset{n\rightarrow
+\infty }{\overline{\lim }}$ $(\Phi ^{\prime }(u_{n},v_{n})-\Phi ^{\prime
}(u,v),(u_{n}-u,v_{n}-v))\leq 0$, then $(u_{n},v_{n})\rightarrow (u,v)$ in $
X $;

iii) $\Phi ^{\prime }:X\rightarrow X^{\ast }$ is a homeomorphism.
\end{proposition}

\begin{remark}
A proof of a simple version of the above proposition can be found in the
references \cite{e12}, \cite{e17}. In the system setting here, the idea of
proof is essentially the same. For readers' convenience and for
completeness, we present it here.
\end{remark}

\begin{proof}
i) It follows from Proposition 2.2 that $\Phi ^{\prime }$ is continuous and
bounded. For any $\xi ,\eta \in\mathbb{R}^{N},$ we have the following inequalities (see \cite{e12}) from which we can
get the strict monotonicity of $\Phi ^{\prime }$:
\begin{equation}
\left[ (\left\vert \xi \right\vert ^{p-2}\xi -\left\vert \eta \right\vert
^{p-2}\eta )(\xi -\eta )\right] \cdot (\left\vert \xi \right\vert
+\left\vert \eta \right\vert )^{2-p}\geq (p-1)\left\vert \xi -\eta
\right\vert ^{2},1<p<2,  \label{3.2}
\end{equation}
\begin{equation}
(\left\vert \xi \right\vert ^{p-2}\xi -\left\vert \eta \right\vert
^{p-2}\eta )(\xi -\eta )\geq (\frac{1}{2})^{p}\left\vert \xi -\eta
\right\vert ^{p},p\geq 2.  \label{3.1}
\end{equation}

ii) From i), if $u_{n}\rightharpoonup u$ and $\underset{n\rightarrow +\infty
}{\overline{\lim }}(\Phi ^{\prime }(u_{n},v_{n})-\Phi ^{\prime
}(u,v),(u_{n}-u,v_{n}-v))\leq 0$, then
\begin{equation}
\underset{n\rightarrow +\infty }{\lim }(\Phi ^{\prime }(u_{n},v_{n})-\Phi
^{\prime }(u,v),(u_{n}-u,v_{n}-v))=0.  \label{a11}
\end{equation}

We claim that $\nabla u_{n}(x)\rightarrow \nabla u(x)$ in measure.

Denote
\begin{align*}
U& =\{x\in \Omega \mid p(x)\geq 2\},V=\{x\in \Omega \mid 1<p(x)<2\}, \\
V_{n}^{-}& =\{x\in V\mid \left\vert \nabla u_{n}\right\vert +\left\vert
\nabla u\right\vert <1\},V_{n}^{+}=\{x\in V\mid \left\vert \nabla
u_{n}\right\vert +\left\vert \nabla u\right\vert \geq 1\},
\end{align*}
\begin{equation*}
\Phi _{n}=(\left\vert \nabla u_{n}\right\vert ^{p(x)-2}\nabla
u_{n}-\left\vert \nabla u\right\vert ^{p(x)-2}\nabla u)(\nabla u_{n}-\nabla
u),\Psi _{n}=(\left\vert \nabla u_{n}\right\vert +\left\vert \nabla
u\right\vert ).
\end{equation*}

From (\ref{3.1}), we have
\begin{equation} \label{3.1b}
\int_{U}\left\vert \nabla u_{n}-\nabla u\right\vert ^{p(x)}dx\leq
2^{p^{+}}\int_{U}\Phi _{n}dx\leq 2^{p^{+}}\int_{\Omega }\Phi
_{n}dx\rightarrow 0. 
\end{equation}

In view of (\ref{3.2}), we have
\begin{equation} \label{3.1c}
\int_{V_{n}^{-}}\left\vert \nabla u_{n}-\nabla u\right\vert ^{2}dx\leq
\int_{V_{n}^{-}}\frac{1}{p(x)-1}\Phi _{n}dx\leq \frac{1}{p^{-}-1}
\int_{\Omega }\Phi _{n}dx\rightarrow 0.
\end{equation}

Without loss of generality, we may assume that $0<\int_{V_{n}^{+}}\Phi
_{n}dx<1$. Then we have
\begin{align*}
\int_{V_{n}^{+}}\left\vert \nabla u_{n}-\nabla u\right\vert ^{p(x)}dx& \leq
C\int_{V_{n}^{+}}\Phi _{n}^{p(x)/2}\Psi _{n}^{(2-p(x))p(x)/2}dx \\
& =C\Big(\int_{V_{n}^{+}}\Phi _{n}dx\Big)^{1/2}\int_{V_{n}^{+}}\Big(
\int_{V_{n}^{+}}\Phi _{n}dx\Big)^{-1/2}\Phi _{n}^{p(x)/2}\Psi
_{n}^{(2-p(x))p(x)/2}dx \\
& \leq C\Big(\int_{V_{n}^{+}}\Phi _{n}dx\Big)^{1/2}\int_{V_{n}^{+}}\Big[\big(
\int_{V_{n}^{+}}\Phi _{n}dx\big)^{-1/p(x)}\Phi _{n}+\Psi _{n}^{p(x)}\Big]dx
\\
& \leq C\Big(\int_{V_{n}^{+}}\Phi _{n}dx\Big)^{1/2}\Big[1+\int_{V_{n}^{+}}
\Psi _{n}^{p(x)}dx\Big] \\
& =C\Big(\int_{V_{n}^{+}}\Phi _{n}dx\Big)^{1/2}\Big[1+\int_{V_{n}^{+}}\big(
\left\vert \nabla u_{n}\right\vert +\left\vert \nabla u\right\vert \big)
^{p(x)}dx\Big].
\end{align*}

From (\ref{a11}) and the bounded property of $\{u_{n}\}$ in $X$, we have
\begin{equation} \label{3.1d}
\int_{V_{n}^{+}}\left\vert \nabla u_{n}-\nabla u\right\vert ^{p(x)}dx\leq 
\Big(\int_{V_{n}^{+}}\Phi _{n}dx\Big)^{1/2}\Big [1+\int_{V_{n}^{+}}\big(
\left\vert \nabla u_{n}\right\vert +\left\vert \nabla u\right\vert \big)
^{p(x)}dx\Big]\rightarrow 0. 
\end{equation}

Thus $\{\nabla u_{n}\}$ converges in measure to $\nabla u$ in $\Omega $, so
we have by Egorov's Theorem that $\nabla u_{n}(x)\rightarrow \nabla u(x)$
a.e. $x\in \Omega $ up to a subsequence. Consequently, we obtain, by Fatou's
Lemma, that
\begin{equation} \label{3.5}
\underset{n\rightarrow +\infty }{\underline{\lim }}\int_{\Omega }\frac{1}{
p(x)}\left\vert \nabla u_{n}\right\vert ^{p(x)}dx\geq \int_{\Omega }\frac{1}{
p(x)}\left\vert \nabla u\right\vert ^{p(x)}dx. 
\end{equation}

Similarly, we have
\begin{equation}\label{a.1}
\underset{n\rightarrow +\infty }{\underline{\lim }}\int_{\Omega }\frac{1}{
q(x)}\left\vert \nabla v_{n}\right\vert ^{q(x)}dx\geq \int_{\Omega }\frac{1}{
q(x)}\left\vert \nabla v\right\vert ^{q(x)}dx. 
\end{equation}

From $(u_{n},v_{n})\rightharpoonup (u,v)$ in $X$, we have
\begin{equation}\label{3.3}
\underset{n\rightarrow +\infty }{\lim }(\Phi ^{\prime
}(u_{n},v_{n}),(u_{n}-u,v_{n}-v))=\underset{n\rightarrow +\infty }{\lim }
(\Phi ^{\prime }(u_{n},v_{n})-\Phi ^{\prime }(u,v),(u_{n}-u,v_{n}-v))=0.
\end{equation}
We also have
\begin{eqnarray*}
&&(\Phi ^{\prime }(u_{n},v_{n}),(u_{n}-u,v_{n}-v)) \\
&=&\int_{\Omega }\left\vert \nabla u_{n}\right\vert ^{p(x)}dx-\int_{\Omega
}\left\vert \nabla u_{n}\right\vert ^{p(x)-2}\nabla u_{n}\nabla udx \\
&&+\int_{\Omega }\left\vert \nabla v_{n}\right\vert ^{q(x)}dx-\int_{\Omega
}\left\vert \nabla v_{n}\right\vert ^{q(x)-2}\nabla v_{n}\nabla vdx \\
&\geq &\int_{\Omega }\left\vert \nabla u_{n}\right\vert
^{p(x)}dx-\int_{\Omega }\left\vert \nabla u_{n}\right\vert
^{p(x)-1}\left\vert \nabla u\right\vert dx \\
&&+\int_{\Omega }\left\vert \nabla v_{n}\right\vert ^{q(x)}dx-\int_{\Omega
}\left\vert \nabla v_{n}\right\vert ^{q(x)-1}\left\vert \nabla v\right\vert
dx \\
&\geq &\int_{\Omega }\left\vert \nabla u_{n}\right\vert
^{p(x)}dx-\int_{\Omega }\Big(\frac{p(x)-1}{p(x)}\left\vert \nabla
u_{n}\right\vert ^{p(x)}+\frac{1}{p(x)}\left\vert \nabla u\right\vert ^{p(x)}
\Big)dx \\
&&+\int_{\Omega }\left\vert \nabla v_{n}\right\vert ^{q(x)}dx-\int_{\Omega }
\Big( \frac{q(x)-1}{q(x)}\left\vert \nabla v_{n}\right\vert ^{q(x)}+\frac{1}{
q(x)} \left\vert \nabla v\right\vert ^{q(x)}\Big)dx
\end{eqnarray*}

\begin{eqnarray}\label{3.4}
&\geq &\int_{\Omega }\frac{1}{p(x)}\left\vert \nabla u_{n}\right\vert
^{p(x)}dx+\int_{\Omega }\frac{1}{q(x)}\left\vert \nabla v_{n}\right\vert
^{q(x)}dx  \notag \\
&&-\int_{\Omega }\frac{1}{p(x)}\left\vert \nabla u\right\vert
^{p(x)}dx-\int_{\Omega }\frac{1}{q(x)}\left\vert \nabla v\right\vert
^{q(x)}dx. 
\end{eqnarray}
According to (\ref{3.5})-(\ref{3.4}), we obtain
\begin{equation}  \label{3.6}
\underset{n\rightarrow +\infty }{\lim }\int_{\Omega }\frac{1}{p(x)}
\left\vert \nabla u_{n}\right\vert ^{p(x)}dx=\int_{\Omega }\frac{1}{p(x)}
\left\vert \nabla u\right\vert ^{p(x)}dx.\end{equation}

It follows from (\ref{3.6}) that the integrals of the functions family $
\left\{ \frac{1}{p(x)}\left\vert \nabla u_{n}\right\vert ^{p(x)}\right\} $
possess absolute equicontinuity on $\Omega $ (see \cite{j10}, Ch.6, Section
3, Corollary 1, Theorem 4-5). Since
\begin{equation}\label{3.7}
\frac{1}{p(x)}\left\vert \nabla u_{n}(x)-\nabla u(x)\right\vert ^{p(x)}\leq
C(\frac{1}{p(x)}\left\vert \nabla u_{n}(x)\right\vert ^{p(x)}+\frac{1}{p(x)}
\left\vert \nabla u(x)\right\vert ^{p(x)}),
\end{equation}
the integrals of the family $\left\{ \frac{1}{p(x)}\left\vert \nabla
u_{n}(x)-\nabla u(x)\right\vert ^{p(x)}\right\} $ is also absolutely
equicontinuous on $\Omega $ (see \cite{j10}, Ch.6, Section3, Theorem 2) and
therefore
\begin{equation}\label{3.8}
\underset{n\rightarrow +\infty }{\lim }\int_{\Omega }\frac{1}{p(x)}
\left\vert \nabla u_{n}(x)-\nabla u(x)\right\vert ^{p(x)}dx=0. 
\end{equation}
By (\ref{3.8}), we conclude that
\begin{equation} \label{3.9}
\underset{n\rightarrow +\infty }{\lim }\int_{\Omega }\left\vert \nabla
u_{n}(x)-\nabla u(x)\right\vert ^{p(x)}dx=0. 
\end{equation}
From Proposition 2.4 and (\ref{3.9}), we have $u_{n}\rightarrow u$ in $
W_{0}^{1,p(\cdot )}(\Omega )$. Similarly, we have $v_{n}\rightarrow v$ in $
W_{0}^{1,q(\cdot )}(\Omega )$. Therefore, $(u_{n},v_{n})\rightarrow (u,v)$
in $X$. Thus, $\Phi ^{\prime }$ is of type $(S_{+})$.

iii) By the strict monotonicity, $\Phi ^{\prime }$ is an injection. Since
\begin{equation*}
\underset{\left\Vert (u,v)\right\Vert \rightarrow +\infty }{\lim }\frac{
(\Phi ^{\prime }(u,v),(u,v))}{\left\Vert (u,v)\right\Vert }=\underset{
\left\Vert (u,v)\right\Vert \rightarrow +\infty }{\lim }\frac{\int_{\Omega
}\left\vert \nabla u\right\vert ^{p(x)}dx+\int_{\Omega }\left\vert \nabla
v\right\vert ^{q(x)}dx}{\left\Vert (u,v)\right\Vert }=+\infty ,
\end{equation*}
$\Phi ^{\prime }$ is coercive, thus $\Phi ^{\prime }$ is a surjection in
view of Minty-Browder's Theorem (see \cite{j9}, Th.26A). Hence $\Phi
^{\prime } $ has an inverse mapping $(\Phi ^{\prime })^{-1}:X^{\ast
}\rightarrow X$. Therefore, the continuity of $(\Phi ^{\prime })^{-1}$ is
sufficient to ensure $\Phi ^{\prime }$ to be a homeomorphism.

If $f_{n}$, $f\in X^{\ast }$, $f_{n}\rightarrow f$ in $X^{\ast}$, let $
(u_{n},v_{n})=(\Phi ^{\prime })^{-1}(f_{n})$, $(u,v)=(\Phi ^{\prime
})^{-1}(f),$ then $\Phi ^{\prime }(u_{n},v_{n})=f_{n}$, $\Phi ^{\prime
}(u,v)=f$. So $\left\{ (u_{n},v_{n})\right\} $ is bounded in $X$. Without
loss of generality, we can assume that $(u_{n},v_{n})\rightharpoonup
(u_{0},v_{0})$ in $X$. Since $f_{n}\rightarrow f$ in $X^{\ast}$, we have
\begin{equation}
\underset{n\rightarrow +\infty }{\lim }(\Phi ^{\prime }((u_{n},v_{n}))-\Phi
^{\prime }(u_{0},v_{0}),(u_{n}-u_{0},v_{n}-v_{0})=\underset{n\rightarrow
+\infty }{\lim }(f_{n},(u_{n}-u_{0},v_{n}-v_{0}))=0.  \label{3.10}
\end{equation}
Since $\Phi ^{\prime }$ is of type $(S_{+})$, $(u_{n},v_{n})\rightarrow
(u_{0},v_{0}),$ we conclude that $(u_{n},v_{n})\rightarrow (u,v)$ in $X$, so
$(\Phi ^{\prime })^{-1}$ is continuous. The proof of Proposition 2.7 is
complete.
\end{proof}

Denote $B(x_{0},\varepsilon ,\delta ,\theta )=\{x\in \mathbb{R} ^{N}\mid
\delta \leq \left\vert x-x_{0}\right\vert \leq \varepsilon ,\frac{x-x_{0}}{
\left\vert x-x_{0}\right\vert }\cdot \frac{\nabla p(x_{0})}{\left\vert
\nabla p(x_{0})\right\vert }\geq \cos \theta \}$, where $\theta \in (0,\frac{
\pi }{2})$. Then we have

\begin{lemma}
If $p\in C^{1}(\overline{\Omega })$, $x_{0}\in \Omega $ satisfy $\nabla
p(x_{0})\neq 0$, then there exist a positive $\varepsilon $ small enough
such that
\begin{equation}
(x-x_{0})\cdot \nabla p(x)>0,\forall x\in B(x_{0},\varepsilon ,\delta
,\theta )\text{,}  \label{a3}
\end{equation}
and
\begin{equation}
\max \{p(x)\mid x\in \overline{B(x_{0},\varepsilon )}\}=\max \{p(x)\mid x\in
B(x_{0},\varepsilon ,\varepsilon ,\theta )\}.  \label{a4}
\end{equation}
\end{lemma}

\begin{proof}
Since $p\in C^{1}(\overline{\Omega })$, for any $x\in B(x_{0},\varepsilon
,\delta ,\theta )$, when $\varepsilon $ is small enough, it is easy to see
that
\begin{eqnarray*}
\nabla p(x)\cdot (x-x_{0}) &=&(\nabla p(x_{0})+o(1))\cdot (x-x_{0}) \\
&=&\nabla p(x_{0})\cdot (x-x_{0})+o(\left\vert x-x_{0}\right\vert ) \\
&\geq &\left\vert \nabla p(x_{0})\right\vert \left\vert x-x_{0}\right\vert
\cos \theta +o(\left\vert x-x_{0}\right\vert )>0,
\end{eqnarray*}
where $o(1)\in\mathbb{R}^{N}$ is a function and $o(1)\rightarrow 0$ uniformly as $\left\vert
x-x_{0}\right\vert \rightarrow 0$.

When $\varepsilon $ is small enough, (\ref{a3}) is valid. Since $p\in C^{1}(
\overline{\Omega })$, there exist a small enough positive $\varepsilon $
such that
\begin{equation*}
p(x)-p(x_{0})=\nabla p(y)\cdot (x-x_{0})=(\nabla p(x_{0})+o(1))\cdot
(x-x_{0}),\text{ }
\end{equation*}
where $y=x_{0}+\tau (x-x_{0})$ and $\tau \in (0,1)$, $o(1)\in\mathbb{R}^{N}$ is a function and $o(1)\rightarrow 0$ uniformly as $\left\vert
x-x_{0}\right\vert \rightarrow 0$.

Suppose $x\in \overline{B(x_{0},\varepsilon )}\backslash B(x_{0},\varepsilon
,\delta ,\theta )$. Denote $x^{\ast }=x_{0}+\varepsilon \nabla
p(x_{0})/\left\vert \nabla p(x_{0})\right\vert $.

Suppose $\frac{x-x_{0}}{\left\vert x-x_{0}\right\vert }\cdot \frac{\nabla
p(x_{0})}{\left\vert \nabla p(x_{0})\right\vert }<\cos \theta $. When $
\varepsilon $ is small enough, we have
\begin{eqnarray*}
p(x)-p(x_{0}) &=&(\nabla p(x_{0})+o(1))\cdot (x-x_{0}) \\
&<&\left\vert \nabla p(x_{0})\right\vert \left\vert x-x_{0}\right\vert \cos
\theta +o(\varepsilon ) \\
&\leq &(\nabla p(x_{0})+o(1))\cdot \varepsilon \nabla p(x_{0})/\left\vert
\nabla p(x_{0})\right\vert \\
&=&p(x^{\ast })-p(x_{0}),
\end{eqnarray*}
where $o(1)\in\mathbb{R}^{N}$ is a function and $o(1)\rightarrow 0$ as $\varepsilon \rightarrow 0$.

Suppose $\left\vert x-x_{0}\right\vert <\delta $. When $\varepsilon $ is
small enough, we have
\begin{eqnarray*}
p(x)-p(x_{0}) &=&(\nabla p(x_{0})+o(1))\cdot (x-x_{0}) \\
&\leq &\left\vert \nabla p(x_{0})\right\vert \left\vert x-x_{0}\right\vert
+o(\varepsilon ) \\
&<&(\nabla p(x_{0})+o(1))\cdot \varepsilon \nabla p(x_{0})/\left\vert \nabla
p(x_{0})\right\vert \\
&=&p(x^{\ast })-p(x_{0}),
\end{eqnarray*}
where $o(1)\in\mathbb{R}^{N}$ is a function and $o(1)\rightarrow 0$ as $\varepsilon \rightarrow 0$.
Thus
\begin{equation}
\max \{p(x)\mid x\in \overline{B(x_{0},\varepsilon )}\}=\max \{p(x)\mid x\in
B(x_{0},\varepsilon ,\delta ,\theta )\}.  \label{d1}
\end{equation}

It follows from (\ref{a3}) and (\ref{d1}) that (\ref{a4}) is valid. Proof of
Lemma 2.9 is complete.
\end{proof}

\begin{lemma}
Suppose that $F(x,u,v)$ satisfies the following inequality
\begin{equation*}
C_{1}\left\vert u\right\vert ^{p(x)}[\ln (e+\left\vert u\right\vert
)]^{a(x)}+C_{1}\left\vert v\right\vert ^{q(x)}[\ln (e+\left\vert
v\right\vert )]^{b(x)}\leq F(x,u,v),\forall \left\vert u\right\vert
+\left\vert v\right\vert \geq M,\forall x\in \Omega ,
\end{equation*}
where $a(\cdot )>p(\cdot )$ and $b(\cdot )>q(\cdot )$ on $\overline{\Omega }$
, and $x_{1},x_{2}\in \Omega $ are two different points such that $\nabla
p(x_{1})\neq 0$ and $\nabla q(x_{2})\neq 0$. Let
\begin{equation*}
h_{1}(x)=\left\{
\begin{array}{cc}
0, & \left\vert x-x_{1}\right\vert >\varepsilon \\
\varepsilon -\left\vert x-x_{1}\right\vert , & \left\vert x-x_{1}\right\vert
\leq \varepsilon
\end{array}
\right. ,
\end{equation*}
and

\begin{equation*}
h_{2}(x)=\left\{
\begin{array}{cc}
0, & \left\vert x-x_{2}\right\vert >\varepsilon \\
\varepsilon -\left\vert x-x_{2}\right\vert , & \left\vert x-x_{2}\right\vert
\leq \varepsilon
\end{array}
\right. ,
\end{equation*}
where $\varepsilon $ is as defined in Lemma 2.9 and small enough such that $
\varepsilon <\left\vert x_{2}-x_{1}\right\vert $. Then there holds
\begin{equation*}
\int_{\Omega }\left\vert \nabla t\,h_{1}\right\vert ^{p(x)}dx+\int_{\Omega
}\left\vert \nabla t\,h_{2}\right\vert ^{q(x)}dx-\int_{\Omega }\lambda
\left\vert t\,h_{1}\right\vert ^{\alpha (x)}\left\vert t\,h_{2}\right\vert
^{\beta (x)}+F(x,t\,h_{1},t\,h_{2})dx\rightarrow -\infty ,
\end{equation*}
as $t\rightarrow +\infty $.
\end{lemma}

\begin{proof}
According to $(H_{\alpha ,\beta })$, there is a constant $\theta \in (0,1)$
such that
\begin{equation*}
\frac{\alpha (x)}{\theta p(x)}+\frac{\beta (x)}{\theta q(x)}\leq 1,\forall
x\in \overline{\Omega }.
\end{equation*}
Therefore, we have
\begin{equation*}
\left\vert u\right\vert ^{\alpha (x)}\left\vert v\right\vert ^{\beta
(x)}\leq \frac{\alpha (x)}{\theta p(x)}\left\vert u\right\vert ^{\alpha (x)
\frac{\theta p(x)}{\alpha (x)}}+(\frac{\theta p(x)}{\alpha (x)}
)^{0}\left\vert v\right\vert ^{\beta (x)(\frac{\theta p(x)}{\alpha (x)}
)^{0}}\leq \left\vert u\right\vert ^{\theta p(x)}+\left\vert v\right\vert
^{\theta q(x)}+1.
\end{equation*}

To complete the proof of this lemma, it is sufficient to show that
\begin{equation*}
G(th_1):=\int_{\Omega }\frac{1}{p(x)}\left\vert \nabla th_{1}\right\vert
^{p(x)}-\int_{\Omega }C_{1}\left\vert th_{1}\right\vert ^{p(x)}[\ln
(e+\left\vert th_{1}\right\vert )]^{a(x)}dx\rightarrow -\infty \text{ as }
t\rightarrow +\infty .
\end{equation*}

It is easy to see the following two inequalities hold:
\begin{equation*}
\int_{\Omega }\frac{1}{p(x)}\left\vert \nabla th_{1}\right\vert
^{p(x)}dx\leq C_{2}\int_{B(x_{0},\varepsilon ,\delta ,\theta )}\left\vert
\nabla th_{1}\right\vert ^{p(x)}dx,
\end{equation*}
\begin{equation*}
\int_{\Omega }C_{1}\left\vert th_{1}\right\vert ^{p(x)}[\ln (e+\left\vert
th_{1}\right\vert )]^{a(x)}dx\geq \int_{B(x_{0},\varepsilon ,\delta ,\theta
)}C_{1}\left\vert th_{1}\right\vert ^{p(x)}[\ln (e+\left\vert
th_{1}\right\vert )]^{a(x)}dx.
\end{equation*}

To proceed, we shall use polar coordinates. Let $r=\left\vert
x-x_{0}\right\vert $. Since $p\in C^{1}(\overline{\Omega })$, it follows
from (\ref{a3}) that there exist positive constants $c_{1}$ and $c_{2}$ such
that
\begin{equation*}
p(\varepsilon ,\omega )-c_{2}(\varepsilon -r)\leq p(r,\omega )\leq
p(\varepsilon ,\omega )-c_{1}(\varepsilon -r),\forall (r,\omega )\in
B(x_{0},\varepsilon ,\delta ,\theta ).
\end{equation*}

Therefore, we have
\begin{eqnarray}
\int_{B(x_{0},\varepsilon ,\delta ,\theta )}\left\vert \nabla
th_{1}\right\vert ^{p(x)}dx &=&\int_{B(x_{0},\varepsilon ,\delta ,\theta
)}t^{p(r,\omega )}r^{N-1}drd\omega  \notag \\
&\leq &\int_{B(x_{0},\varepsilon ,\delta ,\theta )}t^{p(\varepsilon ,\omega
)-c_{1}(\varepsilon -r)}r^{N-1}drd\omega  \notag \\
&\leq &\varepsilon ^{N-1}\int_{B(x_{0},\varepsilon ,\delta ,\theta
)}t^{p(\varepsilon ,\omega )-c_{1}(\varepsilon -r)}drd\omega  \notag \\
&\leq &\varepsilon ^{N-1}\int_{B(x_{0},1,1,\theta )}\frac{t^{p(\varepsilon
,\omega )}}{c_{1}\ln t}d\omega .  \label{a5}
\end{eqnarray}

Since $p\in C^{1}(\overline{\Omega })$ and $a(\cdot )>p(\cdot )$ on $
\overline{\Omega }$, we conclude that, for $\varepsilon $ small enough,
there exists a $\epsilon _{1}>0$ such that
\begin{equation*}
a(x)\geq \max \{p(x)+\epsilon _{1}\mid x\in B(x_{0},\varepsilon ,\delta
,\theta )\}.
\end{equation*}

Thus, when $t$ is large enough, we have
\begin{eqnarray*}
&&\int_{B(x_{0},\varepsilon ,\delta ,\theta )}C_{1}\left\vert
th_{1}\right\vert ^{p(x)}[\ln (e+\left\vert th_{1}\right\vert )]^{a(x)}dx \\
&=&\int_{B(x_{0},\varepsilon ,\delta ,\theta )}C_{1}\left\vert t(\varepsilon
-r)\right\vert ^{p(r,\omega )}r^{N-1}[\ln (e+\left\vert t(\varepsilon
-r)\right\vert )]^{a(r,\omega )}drd\omega \\
&\geq &C_{1}\delta ^{N-1}\int_{B(x_{0},\varepsilon ,\delta ,\theta
)}\left\vert t\right\vert ^{p(\varepsilon ,\omega )-c_{2}(\varepsilon
-r)}\left\vert \varepsilon -r\right\vert ^{p(\varepsilon ,\omega
)-c_{1}(\varepsilon -r)}[\ln (e+\left\vert t(\varepsilon -r)\right\vert
)]^{a(r,\omega )}drd\omega \\
&\geq &C_{1}\delta ^{N-1}\int_{B(x_{0},1,1,\theta )}d\omega \int_{\delta
}^{\varepsilon -\frac{1}{\ln t}}\left\vert t\right\vert ^{p(\varepsilon
,\omega )-c_{2}(\varepsilon -r)}\left\vert \varepsilon -r\right\vert
^{p(\varepsilon ,\omega )-c_{1}(\varepsilon -r)}[\ln (e+\left\vert
t(\varepsilon -r)\right\vert )]^{a(r,\omega )}dr \\
&\geq &C_{3}\delta ^{N-1}\int_{B(x_{0},1,1,\theta )}(\frac{1}{\ln t}
)^{p(\varepsilon ,\omega )}[\ln (e+\frac{t}{\ln t})]^{p(\varepsilon ,\omega
)+\epsilon _{1}}\int_{\delta }^{\varepsilon -\frac{1}{\ln t}}\left\vert
t\right\vert ^{p(\varepsilon ,\omega )-c_{2}(\varepsilon -r)}drd\omega \\
&\geq &C_{4}\delta ^{N-1}\int_{B(x_{0},1,1,\theta )}(\ln t)^{\epsilon _{1}}
\frac{\left\vert t\right\vert ^{p(\varepsilon ,\omega )-\frac{c_{2}}{\ln t}}
}{c_{2}\ln t}d\omega \\
&\geq &(\ln t)^{\epsilon _{1}}C_{5}\delta ^{N-1}\int_{B(x_{0},1,1,\theta )}
\frac{\left\vert t\right\vert ^{p(\varepsilon ,\omega )}}{c_{2}\ln t}d\omega.
\end{eqnarray*}

Hence, we have
\begin{equation} \label{a6}
\int_{B(x_{0},\varepsilon ,\delta ,\theta )}C_{1}\left\vert
th_{1}\right\vert ^{p(x)}[\ln (e+\left\vert t\,h_{1}\right\vert
)]^{a(x)}dx\geq (\ln t)^{\epsilon _{1}}C_{5}\int_{B(x_{0},1,1,\theta )}\frac{
\left\vert t\right\vert ^{p(\varepsilon ,\omega )}}{\ln t}d\omega \text{ as }
t\rightarrow +\infty . \end{equation}

It follows from (\ref{a5}) and (\ref{a6}) that $G(t\, h_{1})\rightarrow -\infty
$. The proof of Lemma 2.10 is complete.
\end{proof}

We have the following simple proposition concerning the growth rate of the
nonlinearity $F(\cdot, \cdot, \cdot)$.

\begin{proposition}
(see \cite{e17}) (i) If $F$ satisfies
\begin{equation*}
0<F(x,s,t)\leq \frac{1}{\theta _{1}}sF_{s}(x,s,t)+\frac{1}{\theta _{2}}
tF_{t}(x,s,t)\ \text{for }x\in \overline{\Omega }\text{ and }\left\vert
s\right\vert ^{\theta _{1}}+\left\vert t\right\vert ^{\theta _{2}}\geq 2M,
\end{equation*}
then $F(x,s,t)\geq c_{1}[(\left\vert s\right\vert ^{\theta _{1}}+\left\vert
t\right\vert ^{\theta _{2}})-1],\forall (x,s,t)\in \overline{\Omega }\times
\mathbb{R} \times \mathbb{R}$.
\end{proposition}

\section{Proofs of main results}

With the preparations in the last section, we will in this section give our
proofs of the main results. To be rigorous, we first give the definition of
a weak solution to the problem $(P)$.

\begin{definition}
(i) We call $\left( u,v\right) \in X$ is a weak solution of $(P)$ if
\begin{eqnarray*}
\int_{\Omega }\left\vert \nabla u\right\vert ^{p(x)-2}\nabla u\cdot \nabla
\phi dx &=&\int_{\Omega }(\lambda \alpha (x)\left\vert u\right\vert ^{\alpha
(x)-2}u\left\vert v\right\vert ^{\beta (x)}+F_{u}(x,u,v))\phi dx,\forall
\phi \in W_{0}^{1,p(\cdot )}(\Omega ), \\
\int_{\Omega }\left\vert \nabla v\right\vert ^{q(x)-2}\nabla v\cdot \nabla
\psi dx &=&\int_{\Omega }(\lambda \beta (x)\left\vert u\right\vert ^{\alpha
(x)}\left\vert v\right\vert ^{\beta (x)-2}v+F_{v}(x,u,v))\psi dx,\forall
\psi \in W_{0}^{1,q(\cdot )}(\Omega ).
\end{eqnarray*}
\end{definition}

The corresponding functional of $(P)$ is given by $\varphi=\varphi(u, v)$
defined below on $X$:
\begin{eqnarray*}
\varphi \left( u,v\right) &=&\Phi (u,v)-\Psi (u,v) \\
&=&\int_{\Omega }\frac{1}{p(x)}\left\vert \nabla u\right\vert ^{p(x)}+\frac{1
}{q(x)}\left\vert \nabla v\right\vert ^{q(x)}dx-\int_{\Omega }\lambda
\left\vert u\right\vert ^{\alpha (x)}\left\vert v\right\vert ^{\beta
(x)}+F(x,u,v)dx,\forall \left( u,v\right) \in X.
\end{eqnarray*}

As compactness is crucial in showing the existence of weak solutions via
critical point theory. We shall introduce the type of compactness which we
shall use in the current study, i.e., the Cerami compactness condition.
\newline

\begin{definition}
\label{Cerami} We say $\varphi $ satisfies Cerami condition in $X$, if any
sequence $\left\{ u_{n}\right\} \subset X$ such that $\left\{ \varphi
(u_{n},v_{n})\right\} $ is bounded and $\left\Vert \varphi ^{\prime
}(u_{n},v_{n})\right\Vert (1+\left\Vert (u_{n},v_{n})\right\Vert
)\rightarrow 0$ as $n\rightarrow +\infty $ has a convergent subsequence.
\end{definition}

It is well-known that the Cerami condition is weaker than the usual
Palais-Samle condition. Under our new growth condition for the system under
investigation, we manage to show that the corresponding functional $\varphi$
satisfies the above Cerami type compactness condition which is sufficient to
yield critical points. More specifically, we have the following lemma.
\newline

\begin{lemma}
If $(H_{\alpha ,\beta })$, $(H_{0})$ and $(H_{1})$ hold, then $\varphi $
satisfies the Cerami condition.
\end{lemma}

\begin{proof}
Let $\{(u_{n},v_{n})\}\subset X$ be a Cerami sequence such that $\varphi
(u_{n},v_{n})\rightarrow c$. From Definition \ref{Cerami}, we know that $
\left\Vert \varphi ^{\prime }(u_{n},v_{n})\right\Vert (1+\left\Vert
(u_{n},v_{n})\right\Vert )\rightarrow 0$ as $n\rightarrow+\infty$. We first
\textit{claim} that to show $\varphi$ satisfies the Cerami condition, it is
sufficient to show that the Cerami sequence $\{(u_{n},v_{n})\}$ is bounded
in $X$. Indeed, suppose $\{(u_{n},v_{n})\}$ is bounded, then $
\{(u_{n},v_{n})\}$ admits a weakly convergent subsequence in $X$. Without
loss of generality, we assume that $(u_{n},v_{n})\rightharpoonup (u,v)$ in $
X $, then $\Psi ^{\prime }(u_{n},v_{n})\rightarrow \Psi ^{\prime }(u,v)$ in $
X^{\ast }$. Since $\varphi ^{\prime }(u_{n},v_{n})=\Phi ^{\prime
}(u_{n},v_{n})-\Psi ^{\prime }(u_{n},v_{n})\rightarrow 0$ in $X^{\ast }$, we
have $\Phi ^{\prime }(u_{n},v_{n})\rightarrow \Psi ^{\prime }(u,v)$ in $
X^{\ast }$. Since $\Phi ^{\prime }$ is a homeomorphism, we have $
(u_{n},v_{n})\rightarrow (u,v)$, hence $\varphi $ satisfies Cerami
condition. Therefore, our claim holds.

Now we show that each Cerami sequence $\{(u_{n},v_{n})\}$ is bounded in $X$.
We argue by contradiction. Suppose not, then up to a subsequence (still
denoted by $\{(u_{n},v_{n})\}$), we have
\begin{equation}  \label{contradiction}
\varphi (u_{n},v_{n})\rightarrow c\text{, }\left\Vert \varphi ^{\prime
}(u_{n},v_{n})\right\Vert (1+\left\Vert (u_{n},v_{n})\right\Vert
)\rightarrow 0\text{, }\left\Vert (u_{n},v_{n})\right\Vert \rightarrow
+\infty .
\end{equation}

Obviously, we have
\begin{eqnarray*}
\left\vert \frac{1}{p(x)}u_{n}\right\vert _{p(\cdot )} &\leq &\frac{1}{p^{-}}
\left\vert u_{n}\right\vert _{p(\cdot )},\left\vert \nabla \frac{1}{p(x)}
u_{n}\right\vert _{p(\cdot )}\leq \frac{1}{p^{-}}\left\vert \nabla
u_{n}\right\vert _{p(\cdot )}+C\left\vert u_{n}\right\vert _{p(\cdot )}, \\
\left\vert \frac{1}{q(x)}v_{n}\right\vert _{q(\cdot )} &\leq &\frac{1}{q^{-}}
\left\vert v_{n}\right\vert _{q(\cdot )},\left\vert \nabla \frac{1}{q(x)}
v_{n}\right\vert _{q(\cdot )}\leq \frac{1}{q^{-}}\left\vert \nabla
v_{n}\right\vert _{q(\cdot )}+C\left\vert v_{n}\right\vert _{q(\cdot )}.
\end{eqnarray*}

From the above inequalities, we easily see that $\left\Vert (\frac{1}{p(x)}
u_{n},\frac{1}{q(x)}v_{n})\right\Vert \leq C\left\Vert
(u_{n},v_{n})\right\Vert $. Therefore, we have $(\varphi ^{\prime
}(u_{n},v_{n}),(\frac{1}{p(x)}u_{n},\frac{1}{q(x)}v_{n}))\rightarrow 0$. We
may assume that
\begin{eqnarray*}
c+1 &\geq &\varphi (u_{n},v_{n})-(\varphi ^{\prime }(u_{n},v_{n}),(\frac{1}{
p(x)}u_{n},\frac{1}{q(x)}v_{n})) \\
&=&\int_{\Omega }\left( \frac{1}{p(x)}\left\vert \nabla u_{n}\right\vert
^{p(x)}+\frac{1}{q(x)}\left\vert \nabla v_{n}\right\vert ^{q(x)}\right)
dx-\int_{\Omega }F(x,u_{n},v_{n})dx \\
&&-\{\int_{\Omega }\frac{1}{p(x)}\left\vert \nabla u_{n}\right\vert
^{p(x)}dx-\int_{\Omega }\frac{1}{p(x)}F_{u}(x,u_{n},v_{n})u_{n}dx-\int_{
\Omega }\frac{1}{p^{2}(x)}u_{n}\left\vert \nabla u_{n}\right\vert
^{p(x)-2}\nabla u_{n}\nabla pdx\} \\
&&-\{\int_{\Omega }\frac{1}{q(x)}\left\vert \nabla v_{n}\right\vert
^{q(x)}dx-\int_{\Omega }\frac{1}{q(x)}F_{v}(x,u_{n},v_{n})v_{n}dx-\int_{
\Omega }\frac{1}{q^{2}(x)}v_{n}\left\vert \nabla v_{n}\right\vert
^{q(x)-2}\nabla v_{n}\nabla qdx\} \\
&&+\int_{\Omega }\lambda (\frac{\alpha (x)}{p(x)}+\frac{\beta (x)}{q(x)}
-1)\left\vert u_{n}\right\vert ^{\alpha (x)}\left\vert v_{n}\right\vert
^{\beta (x)}dx \\
&=&\int_{\Omega }\frac{1}{p^{2}(x)}u_{n}\left\vert \nabla u_{n}\right\vert
^{p(x)-2}\nabla u_{n}\nabla pdx+\int_{\Omega }\frac{1}{q^{2}(x)}
v_{n}\left\vert \nabla v_{n}\right\vert ^{q(x)-2}\nabla v_{n}\nabla qdx \\
&&+\int_{\Omega }\{\frac{1}{p(x)}F_{u}(x,u_{n},v_{n})u_{n}+\frac{1}{q(x)}
F_{v}(x,u_{n},v_{n})v_{n}dx-F(x,u_{n},v_{n})\}dx \\
&&+\int_{\Omega }\lambda (\frac{\alpha (x)}{p(x)}+\frac{\beta (x)}{q(x)}
-1)\left\vert u_{n}\right\vert ^{\alpha (x)}\left\vert v_{n}\right\vert
^{\beta (x)}dx.
\end{eqnarray*}

Hence, there holds
\begin{eqnarray}
&&\int_{\Omega }\{\frac{1}{p(x)}F_{u}(x,u_{n},v_{n})u_{n}+\frac{1}{q(x)}
F_{v}(x,u_{n},v_{n})v_{n}dx-F(x,u_{n},v_{n})\}dx  \notag \\
&\leq &C_{1}(\int_{\Omega }\left\vert u_{n}\right\vert \left\vert \nabla
u_{n}\right\vert ^{p(x)-1}dx+\int_{\Omega }\left\vert v_{n}\right\vert
\left\vert \nabla v_{n}\right\vert ^{q(x)-1}dx+1)  \notag \\
&&+\int_{\Omega }\lambda (1-\frac{\alpha (x)}{p(x)}-\frac{\beta (x)}{q(x)}
)\left\vert u_{n}\right\vert ^{\alpha (x)}\left\vert v_{n}\right\vert
^{\beta (x)}dx  \notag \\
&\leq &\sigma \int_{\Omega }\frac{\left\vert \nabla u_{n}\right\vert ^{p(x)}
}{\ln (e+\left\vert u_{n}\right\vert )}+\frac{\left\vert \nabla
v_{n}\right\vert ^{q(x)}}{\ln (e+\left\vert v_{n}\right\vert )}
dx+\int_{\Omega }\lambda (1-\frac{\alpha (x)}{p(x)}-\frac{\beta (x)}{q(x)}
)\left\vert u_{n}\right\vert ^{\alpha (x)}\left\vert v_{n}\right\vert
^{\beta (x)}dx  \notag \\
&&+C(\sigma )\int_{\Omega }\left\vert u_{n}\right\vert ^{p(x)}[\ln
(e+\left\vert u_{n}\right\vert )]^{p(x)-1}+\left\vert v_{n}\right\vert
^{q(x)}[\ln (e+\left\vert v_{n}\right\vert )]^{q(x)-1}dx+C_{1},  \label{a1}
\end{eqnarray}
where $\sigma>0$ is a sufficiently small constant.

Note that $\frac{u_{n}}{\ln (e+\left\vert u_{n}\right\vert )}\in
W_{0}^{1,p(\cdot )}(\Omega )$, and $\left\Vert \frac{u_{n}}{\ln
(e+\left\vert u_{n}\right\vert )}\right\Vert _{p(\cdot )}\leq
C_{2}\left\Vert u_{n}\right\Vert _{p(\cdot )}$. Choosing $\frac{u_{n}}{\ln
(e+\left\vert u_{n}\right\vert )}$ as a test function, we have
\begin{eqnarray*}
&&\int_{\Omega }F_{u}(x,u_{n},v_{n})\frac{u_{n}}{\ln (e+\left\vert
u_{n}\right\vert )}+\lambda \frac{\alpha (x)}{\ln (e+\left\vert
u_{n}\right\vert )}\left\vert u_{n}\right\vert ^{\alpha (x)}\left\vert
v_{n}\right\vert ^{\beta (x)}dx \\
&=&\int_{\Omega }\left\vert \nabla u_{n}\right\vert ^{p(x)-2}\nabla
u_{n}\nabla \frac{u_{n}}{\ln (e+\left\vert u_{n}\right\vert )}dx+o(1) \\
&=&\int_{\Omega }\frac{\left\vert \nabla u_{n}\right\vert ^{p(x)}}{\ln
(e+\left\vert u_{n}\right\vert )}dx-\int_{\Omega }\frac{\left\vert
u_{n}\right\vert \left\vert \nabla u_{n}\right\vert ^{p(x)}}{(e+\left\vert
u_{n}\right\vert )[\ln (e+\left\vert u_{n}\right\vert )]^{2}}dx+o(1).
\end{eqnarray*}

It is easy to check that $\frac{\left\vert u_{n}\right\vert \left\vert
\nabla u_{n}\right\vert ^{p(x)}}{(e+\left\vert u_{n}\right\vert )[\ln
(e+\left\vert u_{n}\right\vert )]^{2}}\leq \frac{1}{2}\frac{\left\vert
\nabla u_{n}\right\vert ^{p(x)}}{\ln (e+\left\vert u_{n}\right\vert )}$.
Therefore, we have
\begin{eqnarray}
C_{3}\int_{\Omega }\frac{\left\vert \nabla u_{n}\right\vert ^{p(x)}}{\ln
(e+\left\vert u_{n}\right\vert )}dx-C_{4} &\leq &\int_{\Omega }\frac{
F_{u}(x,u_{n},v_{n})u_{n}}{\ln (e+\left\vert u_{n}\right\vert )}+\lambda
\frac{\alpha (x)}{\ln (e+\left\vert u_{n}\right\vert )}\left\vert
u_{n}\right\vert ^{\alpha (x)}\left\vert v_{n}\right\vert ^{\beta (x)}dx
\notag \\
&\leq &C_{5}\int_{\Omega }\frac{\left\vert \nabla u_{n}\right\vert ^{p(x)}}{
\ln (e+\left\vert u_{n}\right\vert )}dx+C_{6}.  \label{a2}
\end{eqnarray}

Similarly, we have
\begin{eqnarray}
C_{7}\int_{\Omega }\frac{\left\vert \nabla v_{n}\right\vert ^{q(x)}}{\ln
(e+\left\vert v_{n}\right\vert )}dx-C_{8} &\leq &\int_{\Omega }\frac{
F_{v}(x,u_{n},v_{n})v_{n}}{\ln (e+\left\vert v_{n}\right\vert )}+\lambda
\frac{\beta (x)}{\ln (e+\left\vert v_{n}\right\vert )}\left\vert
u_{n}\right\vert ^{\alpha (x)}\left\vert v_{n}\right\vert ^{\beta (x)}dx
\notag \\
&\leq &C_{9}\int_{\Omega }\frac{\left\vert \nabla v_{n}\right\vert ^{q(x)}}{
\ln (e+\left\vert v_{n}\right\vert )}dx+C_{10}.  \label{a2a}
\end{eqnarray}
From (\ref{a1}), (\ref{a2}), (\ref{a2a}) and condition $(H_{1})$, we obtain
that
\begin{eqnarray*}
&&\int_{\Omega }\frac{F_{u}(x,u_{n},v_{n})u_{n}}{\ln (e+\left\vert
u_{n}\right\vert )}+\frac{F_{v}(x,u_{n},v_{n})v_{n}}{\ln (e+\left\vert
v_{n}\right\vert )} \\
&&\overset{\text{(H}_{1}\text{)}}{\leq }C_{7}\int_{\Omega }\{\frac{
F_{u}(x,u_{n},v_{n})u_{n}}{p(x)}+\frac{F_{v}(x,u_{n},v_{n})v_{n}}{q(x)}
-F(x,u_{n})\}dx \\
&\leq &C_{7}\{\sigma \int_{\Omega }\frac{\left\vert \nabla u_{n}\right\vert
^{p(x)}}{\ln (e+\left\vert u_{n}\right\vert )}+\frac{\left\vert \nabla
v_{n}\right\vert ^{q(x)}}{\ln (e+\left\vert v_{n}\right\vert )}
dx+\int_{\Omega }\lambda (1-\frac{\alpha (x)}{p(x)}-\frac{\beta (x)}{q(x)}
)\left\vert u_{n}\right\vert ^{\alpha (x)}\left\vert v_{n}\right\vert
^{\beta (x)}dx+C_{8} \\
&&+C(\sigma )\int_{\Omega }\left\vert u_{n}\right\vert ^{p(x)}[\ln
(e+\left\vert u_{n}\right\vert )]^{p(x)-1}+\left\vert v_{n}\right\vert
^{q(x)}[\ln (e+\left\vert v_{n}\right\vert )]^{q(x)-1}dx\} \\
&\leq &C_{7}\sigma \int_{\Omega }\frac{\left\vert \nabla u_{n}\right\vert
^{p(x)}}{\ln (e+\left\vert u_{n}\right\vert )}+\frac{\left\vert \nabla
v_{n}\right\vert ^{q(x)}}{\ln (e+\left\vert v_{n}\right\vert )}
dx+C_{7}\int_{\Omega }\lambda (1-\frac{\alpha (x)}{p(x)}-\frac{\beta (x)}{
q(x)})\left\vert u_{n}\right\vert ^{\alpha (x)}\left\vert v_{n}\right\vert
^{\beta (x)}dx \\
&&+C_{7}C(\sigma )\int_{\Omega }\left\vert u_{n}\right\vert ^{p(x)}[\ln
(e+\left\vert u_{n}\right\vert )]^{p(x)-1}+\left\vert v_{n}\right\vert
^{q(x)}[\ln (e+\left\vert v_{n}\right\vert )]^{q(x)-1}dx+C_{9} \\
&\leq &\frac{1}{2}\int_{\Omega }\frac{F_{u}(x,u_{n},v_{n})u_{n}}{\ln
(e+\left\vert u_{n}\right\vert )}+\frac{F_{v}(x,u_{n},v_{n})v_{n}}{\ln
(e+\left\vert v_{n}\right\vert )}dx+C_{7}\int_{\Omega }\lambda (1-\frac{
\alpha (x)}{p(x)}-\frac{\beta (x)}{q(x)})\left\vert u_{n}\right\vert
^{\alpha }\left\vert v_{n}\right\vert ^{\beta }dx \\
&&+C_{7}C(\sigma )\int_{\Omega }\left\vert u_{n}\right\vert ^{p(x)}[\ln
(e+\left\vert u_{n}\right\vert )]^{p(x)-1}+\left\vert v_{n}\right\vert
^{q(x)}[\ln (e+\left\vert v_{n}\right\vert )]^{q(x)-1}dx+C_{10}.
\end{eqnarray*}

In view of the assumptions $(H_{\alpha ,\beta })$, $(H_{1})$ and the above
inequality, we conclude that
\begin{eqnarray*}
&&\int_{\Omega }\left\vert u_{n}\right\vert ^{p(x)}[\ln (e+\left\vert
u_{n}\right\vert )]^{a(x)-1}+\left\vert v_{n}\right\vert ^{q(x)}[\ln
(e+\left\vert v_{n}\right\vert )]^{b(x)-1}dx \\
&\leq &C_{11}\int_{\Omega }F_{u}(x,u_{n},v_{n})\frac{u_{n}}{\ln
(e+\left\vert u_{n}\right\vert )}+F_{v}(x,u_{n},v_{n})\frac{v_{n}}{\ln
(e+\left\vert v_{n}\right\vert )}dx \\
&\leq &C_{12}\int_{\Omega }\left\vert u_{n}\right\vert ^{p(x)}[\ln
(e+\left\vert u_{n}\right\vert )]^{p(x)-1}+\left\vert v_{n}\right\vert
^{q(x)}[\ln (e+\left\vert v_{n}\right\vert )]^{q(x)-1}dx+C_{12}.
\end{eqnarray*}

Noticing that $a(\cdot )>p(\cdot )$ and $b(\cdot )>q(\cdot )$ on $\overline{
\Omega }$, we can conclude that
\begin{equation*}
\left\{ \int_{\Omega }\left\vert u_{n}\right\vert ^{p(x)}[\ln (e+\left\vert
u_{n}\right\vert )]^{a(x)-1}+\left\vert v_{n}\right\vert ^{q(x)}[\ln
(e+\left\vert v_{n}\right\vert )]^{b(x)-1}dx\right\}
\end{equation*}
is bounded, which further yields that
\begin{equation*}
\left\{ \int_{\Omega }F_{u}(x,u_{n},v_{n})\frac{u_{n}}{\ln (e+\left\vert
u_{n}\right\vert )}+F_{v}(x,u_{n},v_{n})\frac{v_{n}}{\ln (e+\left\vert
v_{n}\right\vert )}dx\right\}
\end{equation*}
and
\begin{equation*}
\left\{ \int_{\Omega }\lambda (\alpha (x)+\beta (x))\left\vert
u_{n}\right\vert ^{\alpha (x)}\left\vert v_{n}\right\vert ^{\beta
(x)}dx\right\}
\end{equation*}
are bounded. Now, it is easy to see that $\left\{ \int_{\Omega }\frac{
\left\vert F_{u}(x,u_{n},v_{n})u_{n}\right\vert }{\ln (e+\left\vert
u_{n}\right\vert )}+\frac{\left\vert F_{v}(x,u_{n},v_{n})v_{n}\right\vert }{
\ln (e+\left\vert v_{n}\right\vert )}dx\right\} $ is bounded.

Let $\varepsilon >0$ satisfy $\varepsilon <\min \{1,p^{-}-1,q^{-}-1,\frac{1}{
p^{\ast +}},\frac{1}{q^{\ast +}},(\frac{p^{\ast }}{\gamma })^{-}-1,(\frac{
q^{\ast }}{\delta })^{-}-1\}$. Since $\left\Vert \varphi ^{\prime
}(u_{n},v_{n})\right\Vert \left\Vert (u_{n},v_{n})\right\Vert \rightarrow 0$
, we have
\begin{eqnarray*}
&&\int_{\Omega }\left\vert \nabla u_{n}\right\vert ^{p(x)}+\left\vert \nabla
v_{n}\right\vert ^{q(x)}dx \\
&=&\int_{\Omega
}F_{u}(x,u_{n},v_{n})u_{n}+F_{v}(x,u_{n},v_{n})v_{n}dx+\int_{\Omega }\lambda
(\alpha (x)+\beta (x))\left\vert u_{n}\right\vert ^{\alpha (x)}\left\vert
v_{n}\right\vert ^{\beta (x)}dx+o(1) \\
&=&\int_{\Omega }\left\vert F_{u}(x,u_{n},v_{n})u_{n}\right\vert
^{\varepsilon }[\ln (e+\left\vert u_{n}\right\vert )]^{1-\varepsilon
}\left\vert F_{u}(x,u_{n},v_{n})\frac{u_{n}}{\ln (e+\left\vert
u_{n}\right\vert )}\right\vert ^{1-\varepsilon }dx \\
&&+\int_{\Omega }\left\vert F_{v}(x,u_{n},v_{n})v_{n}\right\vert
^{\varepsilon }[\ln (e+\left\vert v_{n}\right\vert )]^{1-\varepsilon
}\left\vert F_{v}(x,u_{n},v_{n})\frac{v_{n}}{\ln (e+\left\vert
v_{n}\right\vert )}\right\vert ^{1-\varepsilon }dx+C \\
&\leq &C_{9}(1+\left\Vert u_{n}\right\Vert )^{1+\varepsilon }\int_{\Omega }
\left[ \frac{[\left\vert F_{u}(x,u_{n},v_{n})u_{n}\right\vert
]^{1+\varepsilon }+C_{10}}{(1+\left\Vert (u_{n},v_{n})\right\Vert )^{\frac{
1+\varepsilon }{\varepsilon }}}\right] ^{\varepsilon }\left[ \frac{
\left\vert F_{u}(x,u_{n},v_{n})u_{n}\right\vert }{\ln (e+\left\vert
u_{n}\right\vert )}\right] ^{1-\varepsilon }dx \\
&&+C_{9}(1+\left\Vert v_{n}\right\Vert )^{1+\varepsilon }\int_{\Omega }\left[
\frac{[\left\vert F_{v}(x,u_{n},v_{n})v_{n}\right\vert ]^{1+\varepsilon
}+C_{10}}{(1+\left\Vert (u_{n},v_{n})\right\Vert )^{\frac{1+\varepsilon }{
\varepsilon }}}\right] ^{\varepsilon }\left[ \frac{\left\vert
F_{v}(x,u_{n},v_{n})v_{n}\right\vert }{\ln (e+\left\vert v_{n}\right\vert )}
\right] ^{1-\varepsilon }dx+C \\
&\leq &C_{11}(1+\left\Vert (u_{n},v_{n})\right\Vert )^{1+\varepsilon
}+C_{12}.
\end{eqnarray*}

The above inequality contradicts with \eqref{contradiction}. Therefore, we
can conclude that $\{(u_{n},v_{n})\}$ is bounded, and the proof of Lemma 3.3
is complete.
\end{proof}

Now we are in a position to give a proof of Theorem 1.1.

Denote $F^{++}(x,u,v)=F(x,S(u),S(v))$, where $S(t)=\max \{0,t\}$. For any $
(u,v)\in X$, we say $(u,v)$ belong to the first, the second, the third or
the fourth quadrant of $X$, if $u\geq 0$ and $v\geq 0$, $u\leq 0$ and $v\geq
0$, $u\leq 0$ and $v\leq 0$, $u\geq 0$ and $v\leq 0$, respectively.\newline

\textbf{Proof of Theorem 1.1}.

It is easy to check that $F^{++}(x,s,t)\in C^{1}(\overline{\Omega }\times
\mathbb{R}^{2}, \mathbb{R})$, and
\begin{equation*}
F_{u}^{++}(x,u,v)=F_{u}(x,S(u),S(v)),F_{v}^{++}(x,u,v)=F_{v}(x,S(u),S(v)).
\end{equation*}

Let's consider the following auxiliary problem
\begin{equation*}
\text{$(P^{++})$ } \left\{
\begin{array}{l}
-div(\left\vert \nabla u\right\vert ^{p(x)-2}\nabla u)=\lambda \alpha
(x)\left\vert S(u)\right\vert ^{\alpha (x)-2}S(u)\left\vert S(v)\right\vert
^{\beta (x)}+F_{u}^{++}(x,u,v)\text{ in }\Omega , \\
-div(\left\vert \nabla v\right\vert ^{q(x)-2}\nabla v)=\lambda \beta
(x)\left\vert S(u)\right\vert ^{\alpha (x)}\left\vert S(v)\right\vert
^{\beta (x)-2}S(v)+F_{v}^{++}(x,u,v)\text{ in }\Omega , \\
u=0=v\text{ }\left. {}\right. \text{on }\partial \Omega.
\end{array}
\right.
\end{equation*}

The corresponding functional of problem $(P^{++})$ is
\begin{equation*}
\varphi ^{++}(u,v)=\Phi (u,v)-\Psi ^{++}(u,v),\forall (u,v)\in X,
\end{equation*}
where
\begin{equation*}
\Psi ^{++}(u,v)=\int_{\Omega }\lambda \left\vert S(u)\right\vert ^{\alpha
(x)}\left\vert S(v)\right\vert ^{\beta (x)}+F(x,S(u),S(v))dx,\forall
(u,v)\in X.
\end{equation*}

Let $\sigma >0$ be small enough such that $\sigma \leq \frac{1}{4}\min
\{\lambda _{p(\cdot )},\lambda _{q(\cdot )}\}$. Such a $\sigma $ exits, as $
\lambda _{p(\cdot )}>0$ and $\lambda _{q(\cdot )}>0$ due to Proposition 2.6.
By the assumptions $(H_{0})$ and $(H_{2})$, we have
\begin{equation*}
F(x,u,v)\leq \sigma (\frac{1}{p(x)}\left\vert u\right\vert ^{p(x)}+\frac{1}{
q(x)}\left\vert v\right\vert ^{q(x)})+C(\sigma )(\left\vert u\right\vert
^{\gamma (x)}+\left\vert v\right\vert ^{\delta (x)})\text{, }\forall
(x,t)\in \Omega \times\mathbb{R}.
\end{equation*}

As noticed above, $\lambda _{p(\cdot )},\lambda _{q(\cdot )}>0$ and we have
also by the choice of $\sigma $ that
\begin{eqnarray*}
\int_{\Omega }\frac{1}{p(x)}\left\vert \nabla u\right\vert ^{p(x)}dx-\sigma
\int_{\Omega }\frac{1}{p(x)}\left\vert u\right\vert ^{p(x)}dx &\geq &\frac{3
}{4}\int_{\Omega }\frac{1}{p(x)}\left\vert \nabla u\right\vert ^{p(x)}, \\
\int_{\Omega }\frac{1}{q(x)}\left\vert \nabla v\right\vert ^{q(x)}dx-\sigma
\int_{\Omega }\frac{1}{q(x)}\left\vert v\right\vert ^{q(x)}dx &\geq &\frac{3
}{4}\int_{\Omega }\frac{1}{q(x)}\left\vert \nabla v\right\vert ^{q(x)}.
\end{eqnarray*}

Next, we shall use spatial decomposition technique. We divide the underlying
domain $\Omega $ into disjoint subsets $\Omega _{1},\cdots ,\Omega _{n_{0}}$
such that
\begin{eqnarray*}
\underset{x\in \overline{\Omega _{j}}}{\min }p^{\ast }(x) &>&\underset{x\in
\overline{\Omega _{j}}}{\max }\gamma (x)\geq \underset{x\in \overline{\Omega
_{j}}}{\min }\gamma (x)>\underset{x\in \overline{\Omega _{j}}}{\max }
p(x),j=1,\cdots ,n_{0}, \\
\underset{x\in \overline{\Omega _{j}}}{\min }q^{\ast }(x) &>&\underset{x\in
\overline{\Omega _{j}}}{\max }\delta (x)\geq \underset{x\in \overline{\Omega
_{j}}}{\min }\delta (x)>\underset{x\in \overline{\Omega _{j}}}{\max }
q(x),j=1,\cdots ,n_{0}.
\end{eqnarray*}

In the following, we denote
\begin{equation*}
f_{j}^{-}=\underset{x\in \overline{\Omega _{j}}}{\min }f(x)\text{, }
f_{j}^{+}=\underset{x\in \overline{\Omega _{j}}}{\max }f(x)\text{, }
j=1,\cdots ,n_{0}\text{, }\forall f\in C(\overline{\Omega }),
\end{equation*}
and
\begin{equation*}
\Phi _{\Omega _{j}}(u,v)=\int_{\Omega _{j}}\frac{1}{p(x)}\left\vert \nabla
u\right\vert ^{p(x)}dx+\int_{\Omega _{j}}\frac{1}{q(x)}\left\vert \nabla
v\right\vert ^{q(x)}dx,\forall (u,v)\in X.
\end{equation*}

Denote
\begin{equation*}
\epsilon =\underset{1\leq i\leq n_{0}}{\min }\{\underset{\Omega _{i}}{\inf }
\gamma (x)-\underset{\Omega _{i}}{\sup }p(x),\underset{\Omega _{i}}{\inf }
\delta (x)-\underset{\Omega _{i}}{\sup }q(x)\}.
\end{equation*}

Denote also $\left\Vert u\right\Vert _{p(\cdot ),\Omega _{i}}$ the norm of $
u $ on $\Omega _{i}$, i.e.
\begin{equation*}
\int_{\Omega _{i}}\frac{1}{p(x)}\left\vert \nabla \frac{u}{\left\Vert
u\right\Vert _{\Omega _{i}}}\right\vert ^{p(x)}dx=1.
\end{equation*}

It is easy to see that $\left\Vert u\right\Vert _{p(\cdot ),\Omega _{i}}\leq
C\left\Vert u\right\Vert _{p(\cdot )}$, and there exist $\xi _{i},\eta
_{i}\in \overline{\Omega _{i}}$ such that
\begin{eqnarray*}
\left\vert u\right\vert _{\gamma (\cdot ),\Omega _{i}}^{\gamma (\xi _{i})}
&=&\int_{\Omega _{i}}\left\vert u\right\vert ^{\gamma (x)}dx, \\
\left\Vert u\right\Vert _{p(\cdot ),\Omega _{i}}^{p(\eta _{i})}
&=&\int_{\Omega _{i}}\frac{1}{p(x)}\left\vert \nabla u\right\vert ^{p(x)}dx.
\end{eqnarray*}

When $\left\Vert u\right\Vert _{p(\cdot )}$ is small enough, we have
\begin{eqnarray*}
C(\sigma )\int_{\Omega }\left\vert u\right\vert ^{\gamma (x)}dx &=&C(\sigma )
\underset{i=1}{\overset{n_{0}}{\sum }}\int_{\Omega _{i}}\left\vert
u\right\vert ^{\gamma (x)}dx \\
&=&C(\sigma )\underset{i=1}{\overset{n_{0}}{\sum }}\left\vert u\right\vert
_{\gamma (\cdot ),\Omega _{i}}^{\gamma (\xi _{i})}\text{ (where }\xi _{i}\in
\overline{\Omega _{i}}\text{)} \\
&\leq &C_{1}\underset{i=1}{\overset{n_{0}}{\sum }}\left\Vert u\right\Vert
_{p(\cdot ),\Omega _{i}}^{\gamma (\xi _{i})}\text{ (by Proposition 2.5)} \\
&\leq &C_{2}\left\Vert u\right\Vert _{p(\cdot )}^{\epsilon }\underset{i=1}{
\overset{n_{0}}{\sum }}\left\Vert u\right\Vert _{p(\cdot ),\Omega
_{i}}^{p(\eta _{i})}\text{ (where }\eta _{i}\in \overline{\Omega _{i}}\text{)
} \\
&=&C_{2}\left\Vert u\right\Vert _{p(\cdot )}^{\epsilon }\underset{i=1}{
\overset{n_{0}}{\sum }}\int_{\Omega _{i}}\frac{1}{p(x)}\left\vert \nabla
u\right\vert ^{p(x)}dx \\
&=&C_{2}\left\Vert u\right\Vert _{p(\cdot )}^{\epsilon }\int_{\Omega }\frac{1
}{p(x)}\left\vert \nabla u\right\vert ^{p(x)}dx \\
&\leq &\frac{1}{4}\int_{\Omega }\frac{1}{p(x)}\left\vert \nabla u\right\vert
^{p(x)}dx.
\end{eqnarray*}

Similarly, when $\left\Vert v\right\Vert _{q(\cdot )}$ is small enough, we
have
\begin{equation*}
C(\sigma )\int_{\Omega }\left\vert v\right\vert ^{\delta (x)}dx\leq \frac{1}{
4}\int_{\Omega }\frac{1}{q(x)}\left\vert \nabla v\right\vert ^{q(x)}dx.
\end{equation*}

Thus, when $\left\Vert (u,v)\right\Vert $ is small enough, we have
\begin{equation}\label{aa.1}
\Phi (u,v)-\int_{\Omega }F(x,u,v)dx\geq \frac{1}{2}\int_{\Omega }\frac{1}{
p(x)}\left\vert \nabla u\right\vert ^{p(x)}dx+\frac{1}{2}\int_{\Omega }\frac{
1}{q(x)}\left\vert \nabla v\right\vert ^{q(x)}dx.
\end{equation}

When $\lambda $ is small enough, for any $(u,v)\in X$ with small enough
norm, we have
\begin{eqnarray*}
\varphi ^{++}(u,v) &=&\Phi (u,v)-\Psi ^{++}(u,v) \\
&\geq &\frac{1}{2}\int_{\Omega }\frac{1}{p(x)}\left\vert \nabla u\right\vert
^{p(x)}dx+\frac{1}{2}\int_{\Omega }\frac{1}{q(x)}\left\vert \nabla
v\right\vert ^{q(x)}dx-\int_{\Omega }\lambda \left\vert u\right\vert
^{\alpha (x)}\left\vert v\right\vert ^{\beta (x)}dx, \\
&\geq &\frac{1}{4}\int_{\Omega }\frac{1}{p(x)}\left\vert \nabla u\right\vert
^{p(x)}dx+\frac{1}{4}\int_{\Omega }\frac{1}{q(x)}\left\vert \nabla
v\right\vert ^{q(x)}dx.
\end{eqnarray*}

Therefore, when $\lambda $ is small enough, there exist $r>0$ and $
\varepsilon >0$ such that $\varphi (u,v)\geq \varepsilon >0$ for every $
(u,v)\in X$ and $\left\Vert (u,v)\right\Vert =r$.

Let $\Omega _{0}\subset \Omega$ be an open ball with radius $\varepsilon$.
Notice that $(H_{\alpha ,\beta })$ holds. Let $\varepsilon >0$ be small
enough such that
\begin{eqnarray*}
p_{\overline{\Omega }_{0}}^{-} &:&=\min \{p(x)\mid \overline{\Omega }
_{0}\}>\alpha _{\overline{\Omega }_{0}}^{+}:=\max \{\alpha (x)\mid \overline{
\Omega }_{0}\}, \\
q_{\overline{\Omega }_{0}}^{-} &:&=\min \{q(x)\mid \overline{\Omega }
_{0}\}>\beta _{\overline{\Omega }_{0}}^{+}:=\max \{\beta (x)\mid \overline{
\Omega }_{0}\},
\end{eqnarray*}
and
\begin{equation*}
\frac{\alpha _{\overline{\Omega }_{0}}^{+}}{p_{\overline{\Omega }_{0}}^{-}}+
\frac{\beta _{\overline{\Omega }_{0}}^{+}}{q_{\overline{\Omega }_{0}}^{-}}<1.
\end{equation*}

Now we pick two functions $u_{0},v_{0}\in C_{0}^{2}(\overline{\Omega }_{0})$
that are positive in $\Omega _{0}$. From $(H_{\alpha ,\beta })$, it is easy
to see that
\begin{eqnarray*}
&&\varphi ^{++}(t^{\frac{1}{p_{\overline{\Omega }_{0}}^{-}}}u_{0},t^{\frac{1
}{q_{\overline{\Omega }_{0}}^{-}}}v_{0}) \\
&=&\Phi (t^{\frac{1}{p_{\overline{\Omega }_{0}}^{-}}}u_{0},t^{\frac{1}{q_{
\overline{\Omega }_{0}}^{-}}}v_{0})-\Psi ^{++}(t^{\frac{1}{p_{\overline{
\Omega }_{0}}^{-}}}u_{0},t^{\frac{1}{q_{\overline{\Omega }_{0}}^{-}}}v_{0})
\\
&\leq &\Phi (t^{\frac{1}{p_{\overline{\Omega }_{0}}^{-}}}u_{0},t^{\frac{1}{
q_{\overline{\Omega }_{0}}^{-}}}v_{0})+2\int_{\Omega }|t^{\frac{1}{p_{
\overline{\Omega }_{0}}^{-}}}u_{0}|^{p(x)}+|t^{\frac{1}{q_{\overline{\Omega }
_{0}}^{-}}}v_{0}|^{q(x)}dx \\
&&-\lambda \int_{\Omega }|t^{\frac{1}{p_{\overline{\Omega }_{0}}^{-}}
}u_{0}|^{\alpha (x)}\cdot |t^{\frac{1}{q_{\overline{\Omega }_{0}}^{-}}
}v_{0}|^{\beta (x)}dx \\
&\leq &t\Phi (u_{0},v_{0})+2t\int_{\Omega
}|u_{0}|^{p(x)}+|v_{0}|^{q(x)}dx-\lambda t^{\frac{\alpha _{\overline{\Omega }
_{0}}^{+}}{p_{\overline{\Omega }_{0}}^{-}}+\frac{\beta _{\overline{\Omega }
_{0}}^{+}}{q_{\overline{\Omega }_{0}}^{-}}}\int_{\Omega }|u_{0}|^{\alpha
(x)}|v_{0}|^{\beta (x)}dx<0\text{ as }t\rightarrow 0^{+}.
\end{eqnarray*}

Thus, $\varphi ^{++}(u,v)$ has at least one nontrivial critical point $
(u_{1}^{\ast },v_{1}^{\ast })$ with $\varphi ^{++}(u_{1}^{\ast },v_{1}^{\ast
})<0$.

From assumption $(H_{3})$, it is easy to see that $(u_{1}^{\ast
},v_{1}^{\ast })$ lies in the first quadrant of $X$. It is easy to see that $
S(-u_{1}^{\ast })\in W_{0}^{1,p(\cdot )}(\Omega )$. Choosing $S(-u_{1}^{\ast
})$ as a test function, we have
\begin{eqnarray*}
&&\int_{\Omega }\left\vert \nabla u_{1}^{\ast }\right\vert ^{p(x)-2}\nabla
u_{1}^{\ast }S(-u_{1}^{\ast })dx \\
&=&\int_{\Omega }[\lambda \alpha (x)\left\vert S(u_{1}^{\ast })\right\vert
^{\alpha (x)-2}S(u_{1}^{\ast })\left\vert S(v_{1}^{\ast })\right\vert
^{\beta (x)}+F_{u}(x,S(u_{1}^{\ast }),S(v_{1}^{\ast }))]S(-u_{1}^{\ast })dx
\\
&=&\int_{\Omega }F_{u}(x,S(u_{1}^{\ast }),S(v_{1}^{\ast }))S(-u_{1}^{\ast
})dx\overset{(H_{3})}{=}0.
\end{eqnarray*}

Thus, $u_{1}^{\ast }\geq 0$. Similarly, we have $v_{1}^{\ast }\geq 0$.
Therefore, $(u_{1}^{\ast },v_{1}^{\ast })$ is a nontrivial solution with
constant sign of $(P)$ and such that $\varphi (u_{1}^{\ast },v_{1}^{\ast
})<0 $. By the former discussion and (\ref{aa.1}), we can see that $
u_{1}^{\ast }$ and $v_{1}^{\ast }$ are both nontrivial. Similarly, we can
see that $(P)$ has a nontrivial $(u_{i}^{\ast },v_{i}^{\ast })$ with
constant sign in the $i $-th quadrant of $X$, such that $\varphi
(u_{i}^{\ast },v_{i}^{\ast })<0$, $i=2,3,4$. Thus $(P)$ has at least four
nontrivial solutions with constant sign. By now, we finished the proof of
Theorem 1.1.

\textbf{Proof of Theorem 1.2}.

According to the proof of Theorem 1.1, when $\lambda $ is small enough,
there exist $r>0$ and $\varepsilon >0$ such that $\varphi ^{++}(u,v)\geq
\varepsilon >0$ for every $(u,v)\in X$ and $\left\Vert (u,v)\right\Vert =r$.

From $(H_{1})$, for any $(x,u,v)\in \overline{\Omega }\times \mathbb{R}
\times \mathbb{R} $, we have
\begin{equation*}
F(x,u,v)\geq C_{1}\left\vert u\right\vert ^{p(x)}[\ln (1+\left\vert
u\right\vert )]^{a(x)}+C_{1}\left\vert v\right\vert ^{q(x)}[\ln
(1+\left\vert v\right\vert )]^{b(x)}-c_{2}.
\end{equation*}

We may assume that there exists two different points $x_{1},x_{2}\in \Omega $
such that $\nabla p(x_{1})\neq 0,\nabla q(x_{2})\neq 0$.

Now we define $h_{1}\in C_{0}(\overline{B(x_{1},\varepsilon )})$ $h_{2}\in
C_{0}(\overline{B(x_{2},\varepsilon )})$ as follows:
\begin{equation*}
h_{1}(x)=\left\{
\begin{array}{cc}
0, & \left\vert x-x_{1}\right\vert \geq \varepsilon \\
\varepsilon -\left\vert x-x_{1}\right\vert , & \left\vert x-x_{1}\right\vert
<\varepsilon
\end{array}
\right. ,
\end{equation*}
\begin{equation*}
h_{2}(x)=\left\{
\begin{array}{cc}
0, & \left\vert x-x_{2}\right\vert \geq \varepsilon \\
\varepsilon -\left\vert x-x_{2}\right\vert , & \left\vert x-x_{2}\right\vert
<\varepsilon
\end{array}
\right. .
\end{equation*}

From Lemma 2.10, we may let $\varepsilon >0$ be small enough such that $
\varepsilon <\frac{1}{2}\left\vert x_{2}-x_{1}\right\vert $ and
\begin{eqnarray*}
\int_{\Omega }\frac{1}{p(x)}\left\vert \nabla th_{1}\right\vert
^{p(x)}-\int_{\Omega }C_{1}\left\vert th_{1}\right\vert ^{p(x)}[\ln
(1+\left\vert th_{1}\right\vert )]^{a(x)}dx &\rightarrow &-\infty \text{ as }
t\rightarrow +\infty \\
\int_{\Omega }\frac{1}{q(x)}\left\vert \nabla th_{2}\right\vert
^{q(x)}-\int_{\Omega }C_{1}\left\vert th_{2}\right\vert ^{q(x)}[\ln
(1+\left\vert th_{2}\right\vert )]^{b(x)}dx &\rightarrow &-\infty \text{ as }
t\rightarrow +\infty .
\end{eqnarray*}
which imply that $\varphi ^{++}(th_{1},th_{2})\rightarrow -\infty $ $($as $
t\rightarrow +\infty )$. Since $\varphi ^{++}\left( 0,0\right) =0,$ $\varphi
^{++}$ satisfies the conditions of the Mountain Pass Lemma. From Lemma 3.3,
we know that $\varphi ^{++}$ satisfies Cerami condition. Therefore, we
conclude that $\varphi ^{++}$ admits at least one nontrivial critical point $
(u_{1},v_{1})$ with $\varphi ^{++}(u_{1},v_{1})>0$. From assumption $(H_{3})$
, we can easily see that $(u_{1},v_{1})$ lies in the first quadrant of $X$.
Thus, $(u_{1},v_{1})$ is a nontrivial solution with constant sign to the
problem $(P)$ in the first quadrant of $X$ satisfying $\varphi
(u_{1},v_{1})>0$.

Similarly, we can see that $(P)$ has a nontrivial solution $(u_{2},v_{2})$
in the third quadrant in $X$, which satisfy $\varphi (u_{2},v_{2})>0$, and (
$u_{1},v_{1}$), ($u_{2},v_{2}$) are all nontrivial. From Theorem 1.1, $(P)$
has nontrivial solutions with constant sign $(u_{i}^{\ast },v_{i}^{\ast })$
in the $i$-th quadrant of $X$ ($i=1,2,3,4$), which satisfies $\varphi
(u_{i}^{\ast },v_{i}^{\ast })<0$. Thus, $(P)$ has at least six nontrivial
solutions with constant sign. By far, we have finished the proof of Theorem
1.2.

Now we proceed to prove Theorem 1.3. For this purpose, we need to do some
preparations. Noticing that $X$ is a reflexive and separable Banach space
(see \cite{22}, Section 17, Theorem 2-3), then there are $\left\{
e_{j}\right\} \subset X$ and $\left\{ e_{j}^{\ast }\right\} \subset X^{\ast
} $ such that
\begin{equation*}
X=\overline{span}\{e_{j}\text{, }j=1,2,\cdots \}\text{, }X^{\ast }=\overline{
span}^{W^{\ast }}\{e_{j}^{\ast }\text{, }j=1,2,\cdots \},
\end{equation*}
and
\begin{equation*}
<e_{j}^{\ast },e_{j}>=\left\{
\begin{array}{c}
1,i=j, \\
0,i\neq j.
\end{array}
\right.
\end{equation*}

For convenience, we write $X_{j}=span\{e_{j}\}$, $Y_{k}=\overset{k}{\underset
{j=1}{\oplus }}X_{j}$, $Z_{k}=\overline{\overset{\infty }{\underset{j=k}{
\oplus }}X_{j}}$.\newline

\begin{lemma}
If $\gamma ,\delta \in C_{+}\left( \overline{\Omega }\right) $, $\gamma
(x)<p^{\ast }(x)$ and $\delta (x)<q^{\ast }(x)$ for any $x\in \overline{
\Omega }$, denote
\begin{equation*}
\beta _{k}=\sup \left\{ \left\vert u\right\vert _{\gamma (\cdot
)}+\left\vert u\right\vert _{\delta (\cdot )}\left\vert \left\Vert
(u,v)\right\Vert =1,(u,v)\in Z_{k}\right. \right\} ,
\end{equation*}
then $\underset{k\rightarrow \infty }{\lim }\beta _{k}=0$.
\end{lemma}

\begin{proof}
Obviously, $0<\beta _{k+1}\leq \beta _{k}$, so $\beta _{k}\rightarrow \beta
\geq 0$. Let $u_{k}\in Z_{k}$ satisfy
\begin{equation*}
\left\Vert (u_{k},v_{k})\right\Vert =1\text{, }0\leq \beta _{k}-\left\vert
u_{k}\right\vert _{\gamma (\cdot )}-\left\vert v_{k}\right\vert _{\delta
(\cdot )}<\frac{1}{k}.
\end{equation*}
Then there exists a subsequence of $\{(u_{k},v_{k})\}$ (which we still
denote by $(u_{k},v_{k})$) such that $(u_{k},v_{k})\rightharpoonup (u,v)$,
and
\begin{equation*}
<e_{j}^{\ast },(u,v)>=\underset{k\rightarrow \infty }{\lim }\left\langle
e_{j}^{\ast },(u_{k},v_{k})\right\rangle =0\text{, }\forall e_{j}^{\ast },
\end{equation*}
which implies that $(u,v)=(0,0)$, and so $(u_{k},v_{k})\rightharpoonup
(0,0). $ Since the imbedding from $W_{0}^{1,p(\cdot )}\left( \Omega \right) $
to $L^{\gamma (\cdot )}\left( \Omega \right) $ is compact, then $
u_{k}\rightarrow 0$ in $L^{\gamma (\cdot )}\left( \Omega \right) $.
Similarly, we have $v_{k}\rightarrow 0$ in $L^{\delta (\cdot )}\left( \Omega
\right) $. Hence we get $\beta _{k}\rightarrow 0$ as $k\rightarrow \infty $.
Proof of Lemma 3.4 is complete.
\end{proof}

In odder to prove Theorem 1.3, we need the following lemma (see in
particular, \cite[Theorem 4.7]{20a}). For a version of this lemma with the
Palais-Samle condition, the (P.S.)-condition, see \cite[P 221, Theorem 3.6]
{e39}.

\begin{lemma}
\label{fountain} Suppose $\varphi \in C^{1}(X, \mathbb{R})$ is even, and
satisfies the Cerami condition. Let $V^{+}$, $V^{-}\subset X$ be closed
subspaces of $X$ with codim$V^{+}+1=$dim $V^{-}$, and suppose there holds

($1^{0}$) $\varphi (0,0)=0$.

($2^{0}$) $\exists \tau >0,$ $\gamma >0$ such that $\forall (u,v)\in V^{+}:$
$\Vert (u,v)\Vert =\gamma \Rightarrow \varphi (u,v)\geq \tau .$

($3^{0}$) $\exists \rho >0\ $such that $\forall (u,v)\in V^{-}:$ $\Vert
(u,v)\Vert \geq \rho \Rightarrow \varphi (u,v)\leq 0.$

Consider the following set:
\begin{equation*}
\Gamma =\{g\in C^{0}(X,X)\mid g\text{ is odd, }g(u,v)=(u,v)\text{ if }
(u,v)\in V^{-}\text{ and }\Vert (u,v)\Vert \geq \rho \},
\end{equation*}
then

($a$) $\forall \delta >0$, $g\in \Gamma $, $S_{\delta }^{+}\cap g(V^{-})\neq
\varnothing $, here $S_{\delta }^{+}=\{(u,v)\in V^{+}\mid \Vert (u,v)\Vert
=\delta \};$

($b$) the number $\varpi :=\underset{g\in \Gamma }{\inf }\underset{\text{ }
(u,v)\in V^{-}}{\sup }\varphi (g(u,v))\geq \tau >0$ is a critical value for $
\varphi $.
\end{lemma}

\textbf{Proof of Theorem 1.3}.

According to $(H_{\alpha ,\beta })$, $(H_{0})$, $(H_{1})$ and $(H_{4})$, we
know that $\varphi$ is an even functional and satisfies the Cerami
condition. Let $V_{k}^{+}=Z_{k}$, it is a closed linear subspace of $X$ and $
V_{k}^{+}\oplus Y_{k-1}=X$.

We may assume that there exists different points $x_{n},y_{n}\in \Omega $
such that $\nabla p(x_{n})\neq 0,\nabla q(y_{n})\neq 0$. We then define $
h_{n}\in C_{0}(\overline{B(x_{n},\varepsilon _{n})})$ and $h_{n}^{\ast }\in
C_{0}(\overline{B(y_{n},\varepsilon _{n})})$ as follows:
\begin{equation*}
h_{n}(x)=\left\{
\begin{array}{cc}
0, & \left\vert x-x_{n}\right\vert \geq \varepsilon _{n} \\
\varepsilon _{n}-\left\vert x-x_{n}\right\vert , & \left\vert
x-x_{n}\right\vert <\varepsilon _{n}
\end{array}
\right. ,
\end{equation*}
\begin{equation*}
h_{n}^{\ast }(x)=\left\{
\begin{array}{cc}
0, & \left\vert x-y_{n}\right\vert \geq \varepsilon _{n} \\
\varepsilon _{n}-\left\vert x-y_{n}\right\vert , & \left\vert
x-y_{n}\right\vert <\varepsilon _{n}
\end{array}
\right. .
\end{equation*}

Without loss of generality, we may assume that
\begin{equation*}
supp\,h_{i}\cap supp\,h_{i}^{\ast }=\varnothing \text{ for }i=1,2,\cdots
\end{equation*}
and
\begin{equation*}
supp\,h_{i}\cap supp\,h_{j}=\varnothing \text{, }supp\,h_{i}^{\ast }\cap
supp\,h_{j}^{\ast }=\varnothing \text{, }\forall i\neq j.
\end{equation*}

From Lemma 2.9, we may let $\varepsilon _{n}>0$ be small enough such that
\begin{eqnarray*}
\int_{\Omega }\frac{1}{p(x)}\left\vert \nabla t\, h_{n}\right\vert
^{p(x)}-\int_{\Omega }C_{1}\left\vert th_{n}\right\vert ^{p(x)}[\ln
(1+\left\vert t\, h_{n}\right\vert )]^{a(x)}dx &\rightarrow &-\infty \text{ as }
t\rightarrow +\infty , \\
\int_{\Omega }\frac{1}{q(x)}\left\vert \nabla t\,h_{n}^{\ast }\right\vert
^{q(x)}-\int_{\Omega }C_{1}\left\vert t\, h_{n}^{\ast }\right\vert ^{q(x)}[\ln
(1+\left\vert t\, h_{n}^{\ast }\right\vert )]^{b(x)}dx &\rightarrow &-\infty
\text{ as }t\rightarrow +\infty ,
\end{eqnarray*}
which imply that $\varphi (t\, h_{n},t\, h_{n}^{\ast })\rightarrow -\infty $ (as $
t\rightarrow +\infty $).

Set $V_{k}^{-}=span\{(h_{1},h_{1}^{\ast }),\cdots ,(h_{k},h_{k}^{\ast })\}$.
We will prove that there are infinitely many pairs of $V_{k}^{+}$ and $
V_{k}^{-}$, such that $\varphi $ satisfies the conditions of Lemma 3.5 and
that the corresponding critical value $\varpi _{k}:=\underset{g\in \Gamma }{
\inf }\underset{\text{ }(u,v)\in V_{k}^{-}}{\sup }\varphi
(g(u,v))\rightarrow +\infty $ when $k\rightarrow +\infty $, which implies
that there are infinitely many pairs of solutions to the problem $(P)$.

For any $k=1,2,\cdots $, we shall show that there exist $\rho _{k}>\gamma
_{k}>0$ and large enough $k$ such that
\begin{eqnarray*}
(A_{1})\left. {}\right. \text{ }b_{k} &:&=\inf \left\{ \varphi (u,v)\mid
(u,v)\in V_{k}^{+},\left\Vert (u,v)\right\Vert =\gamma _{k}\right\}
\rightarrow +\infty \text{ }(k\rightarrow +\infty ); \\
(A_{2})\left. {}\right. \text{ }a_{k} &:&=\max \left\{ \varphi (u,v)\mid
(u,v)\in V_{k}^{-},\left\Vert (u,v)\right\Vert =\rho _{k}\right\} \leq 0.
\end{eqnarray*}

First, we prove that ($A_{1}$) holds. By direct computations, we have, for
any $(u,v)\in Z_{k}$ with $\left\Vert (u,v)\right\Vert =\gamma
_{k}=(2p^{+}q^{+}C\beta _{k})^{1/(\min \{p^{-},q^{-}\}-\max \{\gamma
^{+},\delta ^{+}\})}$, we have
\begin{equation*}
\begin{array}{l}
\varphi (u,v)=\int_{\Omega }\frac{1}{p(x)}\left\vert \nabla u\right\vert
^{p(x)}dx+\int_{\Omega }\frac{1}{q(x)}\left\vert \nabla v\right\vert
^{q(x)}dx \\
-\int_{\Omega }\lambda \left\vert u\right\vert ^{\alpha (x)}\left\vert
v\right\vert ^{\beta (x)}dx-\int_{\Omega }F(x,u,v)dx \\
\\
\geq \frac{1}{p^{+}}\int_{\Omega }\left\vert \nabla u\right\vert ^{p(x)}dx+
\frac{1}{q^{+}}\int_{\Omega }\left\vert \nabla v\right\vert ^{q(x)}dx \\
-C\int_{\Omega }\left\vert u\right\vert ^{\gamma (x)}dx-C\int_{\Omega
}\left\vert v\right\vert ^{\delta (x)}dx-C_{1} \\
\\
\geq \frac{1}{p^{+}}\left\Vert u\right\Vert _{p(\cdot )}^{p^{-}}-C\left\vert
u\right\vert _{\gamma (\cdot )}^{\gamma (\xi )}+\frac{1}{q^{+}}\left\Vert
v\right\Vert _{q(\cdot )}^{q^{-}}-C\left\vert v\right\vert _{\delta (\cdot
)}^{\delta (\eta )}-C_{1}\text{ (where }\xi ,\eta \in \Omega \text{)} \\
\\
\geq \frac{1}{p^{+}}\left\Vert u\right\Vert _{p(\cdot )}^{p^{-}}-C\beta
_{k}^{\gamma ^{+}}\left\Vert u\right\Vert _{p(\cdot )}^{\gamma ^{+}}+\frac{1
}{q^{+}}\left\Vert v\right\Vert _{q(\cdot )}^{q^{-}}-C\beta _{k}^{\delta
^{-}}\left\Vert v\right\Vert _{q(\cdot )}^{\delta ^{+}}-C_{2} \\
\geq \frac{1}{p^{+}q^{+}}\left\Vert (u,v)\right\Vert ^{\min
\{p^{-},q^{-}\}}-C\beta _{k}\left\Vert (u,v)\right\Vert ^{\max \{\gamma
^{+},\delta ^{+}\}}-C_{2} \\
=\frac{1}{2p^{+}q^{+}}(2p^{+}q^{+}C\beta _{k})^{\min \{p^{-},q^{-}\}/(\min
\{p^{-},q^{-}\}-\max \{\gamma ^{+},\delta ^{+}\})}-C_{2}.
\end{array}
\end{equation*}

Therefore $\varphi (u,v)\geq \frac{1}{2p^{+}q^{+}}\gamma _{k}^{\min
\{p^{-},q^{-}\}}-C_{2}$, $\forall (u,v)\in Z_{k}$ with $\left\Vert
(u,v)\right\Vert =\gamma _{k}$, then $b_{k}\rightarrow +\infty
,(k\rightarrow \infty )$. So we have shown that $(A_1)$ holds.

Next we show that ($A_{2}$) holds. From the definition of $
(h_{n},h_{n}^{\ast })$, it is easy to see that
\begin{equation*}
\varphi (th,th^{\ast })\rightarrow -\infty \text{ as }t\rightarrow +\infty ,
\end{equation*}
for any $(h,h^{\ast })\in V_{k}^{-}=span\{(h_{1},h_{1}^{\ast }),\cdots
,(h_{k},h_{k}^{\ast })\}$ with $\parallel (h,h^{\ast })\parallel =1$.
Therefore, $(A_2)$ also holds.

Now, applying Lemma \ref{fountain}, we finish the proof of Theorem 1.3.

\end{document}